\documentclass[reqno]{amsart}
\usepackage[english]{babel}
\usepackage{amssymb,graphicx}
\usepackage[T1,T5]{fontenc}

\newtheorem{theorem}{Theorem}[section]
\newtheorem{proposition}[theorem]{Proposition}
\newtheorem{lemma}[theorem]{Lemma}
\newtheorem{corollary}[theorem]{Corollary}

\theoremstyle{definition}
\newtheorem{definition}[theorem]{Definition}
\newtheorem{example}[theorem]{Example}

\theoremstyle{remark}
\newtheorem*{remark}{Remark}

\numberwithin{equation}{section}

\newcommand{\RE}{\mbox{$\mathbb{R}$}}

\newcommand{\C}{\mbox{$\mathbb{C}$}}

\newcommand{\E}{\mbox{$\mathcal{E}$}}

\newcommand{\Ep}{\mbox{$\mathcal{E}_{p}$}}

\newcommand{\EE}{\operatorname{E}_{w}}
\newcommand{\PP}{\operatorname{P}}
\newcommand{\be}{\beta^{n-m}}
\newcommand{\dd}{{\bf d}}
\newcommand{\HH}{\operatorname{H}_{m}}
\newcommand{\II}{\operatorname{I}}

\begin{document}

\author{Per \AA hag}
\address{Department of Mathematics and Mathematical Statistics\\ Ume\aa \ University\\ SE-901 87 Ume\aa \\ Sweden}
\email{Per.Ahag@math.umu.se}
\author{Rafa\l\ Czy{\.z}}
\address{Institute of Mathematics \\ Faculty of Mathematics and Computer Science \\ Jagiellonian University\\ \L ojasiewicza 6\\ 30-348 Krak\'ow\\ Poland}
\email{Rafal.Czyz@im.uj.edu.pl}\thanks{The second-named author was supported by the Priority Research Area SciMat under the program Excellence Initiative - Research University at the Jagiellonian University in Krak\'ow.}

\dedicatory{We raise our cups to Urban Cegrell, gone but not forgotten, gone but ever here. \\ Until we meet again in Valhalla!}


\title[Geodesics in the space of $m$-subharmonic functions]{Geodesics in the space of $m$-subharmonic functions with bounded energy}

\begin{abstract} With inspiration from the K\"ahler geometry, we introduce a metric structure on the energy class, $\mathcal{E}_{1,m}$, of $m$-subharmonic functions with bounded energy and show that it is complete. After studying how the metric convergence relates to the accepted convergences in this Caffarelli-Nirenberg-Spruck model, we end by constructing geodesics in a subspace of our complete metric space.
\end{abstract}

\maketitle

\begin{center}\bf
\today
\end{center}

\section{Introduction}

In a series of articles~\cite{Hilbert1}--\cite{Hilbert5}, from the beginning of the twentieth century, Hilbert laid the foundation for function spaces when developing a systematic theory to solve integral equations. On the other hand, in 1906, Fr\'{e}chet~\cite{Frechet} introduced abstract metric spaces. Then began a feverish activity among the mathematical community, which resulted in the significant work of Banach~\cite{Banach} from 1920 about complete normed spaces. This article will study the function space of $m$-subharmonic functions with bounded energy, $\mathcal E_{1,m}$, mainly from a metric space viewpoint. However, our starting point and motivation come from the normed vector space standpoint.

For $1\leq m \leq n$, $n>1$, we will always assume that $\Omega$ is a $m$-hyperconvex domain to ensure the existence of enough $m$-subharmonic
functions, except Section~\ref{sec_geodesic} where we shall assume that $\Omega$ is a $(m+1)$-hyperconvex. Recall that the choice of $m$ makes it possible to interpolate between subharmonic and plurisubharmonic functions in the sense that $1$-subharmonic
functions are only classical subharmonic functions, while $n$-subharmonic functions are plurisubharmonic functions. Therefore, the energy classes $\mathcal E_{1,m}$
enjoys the following inclusions:
\[
\mathcal E_{1,n}\subset\mathcal E_{1,n-1}\subset\cdots\subset \mathcal E_{1,1}\, ,
\]
The idea of these interpolation spaces goes back to Caffarelli et al.~\cite{CNS85}. Pluripotential methods were introduced to this model by B\l ocki~\cite{B}, and the
energy classes, $\mathcal E_{1,m}$, by Lu~\cite{L}. These energy classes were modelled from Cegrell's~\cite{cegrell_pc} energy classes in the pluricomplex case. For a historical account and references see e.g.~\cite{AC2019,LN2}.

To study the space $\mathcal E_{1,m}$ from the normed vector space perspective we define $\delta\mathcal E_{1,m}=\mathcal E_{1,m}-\mathcal E_{1,m}$, since $\mathcal E_{1,m}$ is only a convex cone. Then for any $u\in \delta\mathcal E_{1,m}$ define
\[
\|u\|=\inf_{u_1-u_2=u \atop u_1,u_2\in \mathcal{E}_{1,m}}\left(\int_{\Omega} (-(u_1+u_2))\operatorname{H}_m(u_1+u_2)
\right)^{\frac {1}{m+1}}\, ,
\]
where $\operatorname{H}_m$ is the complex operator Hessian operator.  It was proved in~\cite{thien} that $(\delta\mathcal{E}_{1,m}, \|\cdot\|)$ is a Banach space (for the case $m=n$ see~\cite{mod}). However, some unwanted properties occur. For example, take the case $m=n$, and for $a,b<0$ define the following functions in the unit ball in $\mathbb{C}^n$:
\[
u_a(z)=\max (\ln |z|,a) \quad \text{ and }\quad u_b(z)=\max (\ln |z|,b)\, ,
\]
then if $|a-b|<r$, then $|u_a(z)-u_b(z)|<r$, while
\[
\|u_a-u_b\|=(2\pi)^n\left(b-a+2^{n+1}(-b)\right).
\]
To avoid the mentioned complications, we will let ourselves be inspired by~\cite{T1,T2,T3} and introduce the following metric structure. We define $\dd:\mathcal E_{1,m}\times\mathcal E_{1,m}\to \RE$ by
\[
\dd (u,v)=\EE(u)+\EE(v)-2\EE(\PP(u,v)),
\]
where $\PP(u,v)=\sup\{\varphi\in \mathcal E_{1,m}: \varphi\leq \min(u,v)\}^*$ is the rooftop envelope,  and $\EE$ is the Aubin-Mabuchi energy functional. This type of metric topology has extensively been used and studied in K\"ahler geometry, to mention a few works~\cite{BermanDarvasLu, ChenCheng1, ChenCheng2, ChenCheng3, D1,D2,DDL,DR}. After proving some elementary properties of $\dd$, we prove in Theorem~\ref{thm_compmetric} that $(\mathcal E_{1,m},\dd)$ is a complete metric space. As we shall see in Section~\ref{sec_comparison}, the topology generated by the norm $\|\cdot\|$ is
not comparable with the metric topology of $(\mathcal E_{1,m},\dd)$.

Then in Section~\ref{sec_conv}, we proceed with studying the convergence in
$(\mathcal E_{1,m},\dd)$. By using ideas from~\cite{BermanBoucksumGuedjZeriahi, BoucksomEyssidieuxGuedjZeriahi}, where they were used in pluripotential theory, we prove:

\medskip

\begin{enumerate}\itemsep2mm

\item Let $u_j, v_j\in \mathcal E_{1,m}$ be decreasing or increasing sequences converging to $u,v \in\mathcal E_{1,m}$, respectively. Then $\dd(u_j,v_j)\to \dd(u,v)$, as $j\to \infty$. (Proposition~\ref{monotone}).

\item Let $u_j,u\in \mathcal E_{1,m}$. If $\lim_{j\to \infty}\dd(u_j,u)=0$,  then $u_j\to u$ in $L^1(\Omega)$, and $\HH(u_j)\to \HH(u)$ weakly, as  $j\to \infty$. (Corollary~\ref{cor}).

\item Let $u_j,u\in \mathcal E_{1,m}$. Then $\dd(u_j,u)\to 0$ as $j\to \infty$ if, and only if, $u_j\to u$ in $L^1(\Omega)$ and $\EE(u_j)\to \EE(u)$, as $j\to \infty$ (Theorem~\ref{thm_conv}).

\item If $d(u_j,u)\to 0$, then $u_j\to u$ in capacity $\operatorname {cap}_m$ (Proposition~\ref{cap}).

\end{enumerate}

\medskip

Ever since Menger~\cite{Menger} introduced geodesics into abstract metric spaces, it has been a fundamental part of metric geometry. For us, from the K\"ahler geometric viewpoint,  the influential work of Berndtsson~\cite{ber},  Darvas~\cite{D1,D2}, and  Mabuchi~\cite{Mabuchi} are of paramount importance. In Section~\ref{sec_geodesic}, we shall study geodesics in $(\mathcal E_{1,m},\dd)$.   Although it is not in historical chronological order, we have here been inspired by~\cite{DDL,rashkovskii,T4}. The case $m = n$ was shown by Rashovskii in~\cite{rashkovskii}.  For $1\leq m < n$ a problem arise, since if $\mathbb C^n\ni z\to u(z)$ is $m$-subharmonic, then $\mathbb C^{n}\times \mathbb C\ni(z,\lambda)\to u(z)$ need not to be $(m+1)$-subharmonic, it is only $m$-subharmonic. For this reason, we can not directly use Semmes' method~\cite{Semmes92} to construct geodesics. This leads us to consider the subspace, $\widehat{\E}_{1,m}$, of $(\mathcal E_{1,m},\dd)$ containing those functions that are also $(m+1)$-subharmonic. For further information on Semmes' method, we refer to the inspiring monograph written by Guedj, and Zeriahi~\cite[Chapter~15.2]{GZbook}.

Again, it is not our intention to historically describe the history and significance of the use of metric geometry within K\"ahler geometry but to mention a few more articles~\cite{BB1,DDL2,DDL3,DDL4,McCleerey1, McCleerey2, RossNystrom}, and then we refer to~\cite{D3,GZbook,rashkovskii2021} for a better historical account and an overall picture. It is of interest here to mention the work of He and Li~\cite{HeLi} on geometric pluripotential theory on Sasakian manifolds.

\section{Preliminaries}\label{sec_prelim}

 Let $\Omega \subset \C^n$, $n>1$, be a bounded domain, $1\leq m\leq n$, and define $\mathbb C_{(1,1)}$ to be the set of $(1,1)$-forms with constant coefficients. Then set
\[
\Gamma_m=\left\{\alpha\in \mathbb C_{(1,1)}: \alpha\wedge \beta^{n-1}\geq 0, \dots , \alpha^m\wedge \beta ^{n-m}\geq 0   \right\}\, ,
\]
where $\beta=dd^c|z|^2$ is the canonical K\"{a}hler form in $\C^n$.

\begin{definition}\label{m-sh}  Assume that $\Omega \subset \C^n$, $n>1$, is a bounded domain, $1\leq m\leq n$, and let $u$ be a subharmonic function defined  on $\Omega$.
Then we say that $u$ is \emph{$m$-subharmonic}, if the following inequality holds
\[
dd^cu\wedge\alpha_1\wedge\dots\wedge\alpha_{m-1}\wedge\beta^{n-m}\geq 0\, ,
\]
in the sense of currents for all $\alpha_1,\ldots,\alpha_{m-1}\in \Gamma_m$.
\end{definition}

\begin{definition}  Assume that $\Omega \subset \C^n$, $n>1$, is a bounded domain, $1\leq m\leq n$.
We say that $\Omega$ is \emph{$m$-hyperconvex} if it admits an exhaustion function that is negative
and $m$-subharmonic, i.e. the closure of the set $\{z\in\Omega : \varphi(z)<c\}$ is compact in $\Omega$, for every $c\in (-\infty, 0)$.
\end{definition}

For further information about the geometry of $m$-hyperconvex domains, see e.g.~\cite{ACH}. Throughout this paper we shall assume that $\Omega$ is a $m$-hyperconvex domain.  We say
that a $m$-subharmonic function $\varphi$ on  $\Omega$ belongs to:

\begin{itemize}\itemsep2mm
\item[$(i)$] $\E_{0,m}$ if, $\varphi$ is bounded,
\[
\lim_{z\rightarrow\xi} \varphi (z)=0 \quad \text{ for every } \xi\in\partial\Omega\, ,
\]
and
\[
\int_{\Omega} (dd^c\varphi)^m\wedge\beta^{n-m}<\infty\, ;
\]
\item[$(ii)$]  $\E_{1,m}$ if, there exists a decreasing sequence, $\{u_{j}\}$, $u_{j}\in\E_{0,m}$,
that converges pointwise to $u$ on $\Omega$, as $j$ tends to $\infty$, and
\[
\sup_{j} e_{1,m}(u_j)=\sup_{j}\int_{\Omega}(-u_{j})(dd^c u_j)^m\wedge\beta^{n-m}< \infty\, .
\]
\end{itemize}
In~\cite{L,L3}, it was proved that for $u\in \E_{1,m}$ the complex Hessian operator, $\operatorname H_m(u)$, is well-defined, where
\[
\operatorname{H}_m(u)=(dd^cu)^m\wedge\beta^{n-m}\, .
\]

Theorem~\ref{thm_holder} is essential for us when dealing with $\E_{1,m}$.

\begin{theorem}\label{thm_holder} Let $1\leq m\leq n$, and let $\Omega$ be a bounded $m$-hyperconvex domain in $\mathbb C^n$, $n>1$. Let $u_0,u_1,\ldots ,
u_n\in\E_{1,m}$. If $n\geq 2$, then
\begin{multline*}
\int_\Omega (-u_0) dd^c u_1\wedge\cdots\wedge dd^c u_m\wedge \beta^{n-m}\\ \leq
\; e_{1,m}(u_0)^{1/(1+m)}e_{1,m}(u_1)^{1/(1+m)}\cdots
e_{1,m}(u_m)^{1/(1+m)}\, .
\end{multline*}
\end{theorem}
\begin{proof}
See e.g. Lu~\cite{L,L3}, and Nguyen~\cite{thien}. For the case when $m=n$ see Theorem~3.4 in~\cite{persson} (see also~\cite{czyz_energy,cegrell_pc,CP}).
\end{proof}

The following comparison principles will come in handy, see~\cite{L3} for proofs.

\begin{theorem} Let $1\leq m\leq n$, and let $\Omega$ be a bounded $m$-hyperconvex domain in $\mathbb C^n$, $n>1$. Let $u,v\in \mathcal E_{1,m}$. Then
\begin{enumerate}
\item
\[
\int_{\{u<v\}}\HH(v)\leq \int_{\{u<v\}}\HH(u);
\]
\item if $\HH(v)\leq \HH(u)$, then $u\leq v$.
\end{enumerate}
\end{theorem}

We need the following domination principle.

\begin{proposition}\label{dp} Let $1\leq m\leq n$, and let $\Omega$ be a bounded $m$-hyperconvex domain in $\mathbb C^n$, $n>1$.  Let $u,v\in \mathcal E_{1,m}$ are such that $\HH(u)(u<v)=0$, then $u\geq v$.
\end{proposition}
\begin{proof} Let $w_0\in \mathcal E_{0,m}$ be such that $dV_{2n}=\HH(w_0)$. Let $t>0$, then we have by the comparison principle
\begin{multline*}
t^m\int_{\{u<v+tw_0\}}dV_{2n}=t^m\int_{\{u<v+tw_0\}}\HH(w_0)\leq
\int_{\{u<v+tw_0\}}\HH(v+tw_0)\leq\\
\int_{\{u<v+tw_0\}}\HH(u)\leq \int_{\{u<v\}}\HH(u)=0.
\end{multline*}
Therefore, $V_{2n}(\{u<v\})=\lim_{t\to 0}V_{2n}(\{u<v+tw_0\})=0$, so $u\geq v$ a.e. Hence, $u\geq v$.
\end{proof}

The following result can be found in~\cite{L3}.

\begin{proposition}\label{dp2} Let $1\leq m\leq n$, and let $\Omega$ be a bounded $m$-hyperconvex domain in $\mathbb C^n$, $n>1$. Let $u,v\in \mathcal E_{1,m}$ then
\[
\HH(\max(u,v))\geq \chi_{\{u\geq v\}}\HH(u)+\chi_{\{u<v\}}\HH(v).
\]
\end{proposition}

We shall need the following convergence results. We believe that it is well known, but we can not find an appropriate reference.

\begin{proposition}\label{increasing} Let $1\leq m\leq n$, and let $\Omega$ be a bounded $m$-hyperconvex domain in $\mathbb C^n$, $n>1$.
\begin{enumerate}
\item Suppose $u^k\in\mathcal E_1, 1\leq k \leq m+1$. For any sequences $u_j^k\in\mathcal E_1$ that increases to $u^k$,
$j\to\infty$, and $\varphi\in \mathcal E_1$, it holds
\[
\lim_{j\to \infty}\int_{\Omega}(-\varphi)dd^cu_j^1\wedge \dots\wedge dd^cu_j^m\wedge\be=\int_{\Omega}(-\varphi)dd^cu^1\wedge \dots\wedge dd^cu^m\wedge\be.
\]
\item Suppose $u_j, u\in\mathcal E_1$, $j\in \mathbb N$. If $u_j$ increases to $u$,
$j\to\infty$, then
\[
\lim_{j\to \infty}\int_{\Omega}(-u_j)\HH(u_j)=\int_{\Omega}(-u)\HH(u).
\]
\end{enumerate}
\end{proposition}
\begin{proof}
\emph{(1)} Note that the sequence $\int_{\Omega}\varphi dd^cu_j^1\wedge \dots \wedge dd^cu_j^m\wedge \be$ is increasing  and the limit does not
depend on the particular sequences $u_j^k$, see e.g.~\cite{cegrell_bdd,L3}.
If $v_j^k$ is another sequence increasing to $u^k$, we get
\begin{multline*}
\int_{\Omega} \varphi dd^cv_j^1\wedge dd^cv_{j}^2\wedge\cdots \wedge dd^cv_{j}^m\wedge\be=\int_{\Omega}
v_j^1dd^c\varphi\wedge dd^cv_j^2\wedge\cdots \wedge dd^cv_j^m\wedge\be\\
\leq \int_{\Omega} u^1dd^c \varphi\wedge dd^cv_j^2\wedge\cdots \wedge dd^cv_j^m\wedge\be \\
=\lim_{s_1\to \infty}\int_{\Omega} u_{s_1}^{1}dd^c \varphi\wedge dd^cv_j^2\wedge\cdots \wedge dd^cv_j^m\wedge\be \\
=\lim_{s_1\to \infty}\int_{\Omega} v_j^2dd^c\varphi \wedge dd^cu_{s_1}^1\wedge\cdots \wedge dd^cv_j^m\wedge\be  \leq
\ldots \leq \\ \lim_{s_1,\dots,s_m\to \infty}\int_{\Omega} \varphi dd^cu_{s_1}^1\wedge\cdots\wedge dd^cu_{s_m}^m\wedge\be\\
\leq \lim\limits_{s\to\infty}\int_{\Omega} \varphi dd^cu_s^1\wedge dd^cu_s^2\wedge\cdots \wedge dd^cu_s^m\wedge\be.
\end{multline*}
Therefore, $\lim_{j\to \infty}\int \varphi dd^cv_j^1\wedge dd^cv_j^2\wedge\dots \wedge dd^cv_j^m\wedge\be$ exists,
and
\begin{multline*}
\lim_{j\to \infty}\int_{\Omega} \varphi dd^cv_j^1\wedge dd^cv_j^2\wedge\dots \wedge dd^cv_j^m\wedge\be\\ \leq
\lim_{j\to \infty}\int_{\Omega} \varphi dd^cu_j^1\wedge dd^cu_{j}^2\wedge\ldots \wedge dd^cu_j^m\wedge\be.
\end{multline*}
Similarly, one can obtain the inverse inequality. Hence, we can conclude that the limits are equal.

We shall prove, by induction, that for any $k\in\{1,\dots,m\}$
\begin{multline}\label{increas}
\lim_{j\to \infty}\int_{\Omega}(-\varphi)dd^cu_j^1\wedge \dots\wedge dd^cu_j^k\wedge dd^{c}u^{k+1}\wedge\dots\wedge dd^cu^m\wedge\be\\
=\int_{\Omega}(-\varphi)dd^cu^1\wedge \dots\wedge dd^cu^m\wedge\be.
\end{multline}
For $k=1$ we obtain by the monotone convergence theorem
\begin{multline*}
\lim_{j\to \infty}\int_{\Omega}(-\varphi)dd^cu_j^1\wedge dd^cu^2\wedge \dots\wedge dd^cu^m\wedge\be=\\
\lim_{j\to \infty}\int_{\Omega}(-u^1_j)dd^c\varphi\wedge dd^cu^2\wedge \dots\wedge dd^cu^m\wedge\be=\\
\int_{\Omega}(-u^1)dd^c\varphi\wedge dd^cu^2\wedge \dots\wedge dd^cu^m\wedge\be=\\
\int_{\Omega}(-\varphi)dd^cu^1\wedge dd^cu^2\wedge \dots\wedge dd^cu^m\wedge\be.
\end{multline*}

Now suppose that (\ref{increas}) is valid for $k=p$. We show that (\ref{increas}) holds for $k=p+1$.
From our assumption we have
\begin{multline*}
\lim_{j\to \infty}\int_{\Omega}(-\varphi)dd^cu_j^1\wedge \dots\wedge dd^cu_j^p\wedge dd^{c}u^{p+1}\wedge\dots\wedge dd^cu^m\wedge\be\\
=\int_{\Omega}(-\varphi)dd^cu^1\wedge \dots\wedge dd^cu^m\wedge\be
\end{multline*}
so it is enough to prove that
\begin{multline*}
\lim_{j\to \infty}\int_{\Omega}(-\varphi)dd^cu_j^1\wedge\ldots \wedge dd^cu_j^p\wedge dd^cu_j^{p+1}\wedge
dd^cu^m\wedge\be=\\
\lim_{j\to \infty}\int_{\Omega}(-\varphi)dd^cu_j^1\wedge\ldots \wedge dd^cu_j^p\wedge dd^cu^{p+1}\wedge
dd^cu^m\wedge\be\, .
\end{multline*}
Since the limit (\ref{increas}) does not depend on particular sequence then above limits are equal.

\emph{(2)} Note that the sequence $e_1(u_j)=\int_{\Omega}(-u_j)\HH(u_j)$ is decreasing and
\[
\lim_{j\to \infty}e_1(u_j)\geq e_1(u).
\]
Now fix $k\leq j$. Then $e_1(u_j)\leq \int_{\Omega}(-u_k)\HH(u_j)$ and then by the first part of Proposition~\ref{increasing}
\[
\lim_{j\to \infty}e_1(u_j)\leq \lim_{j\to \infty}\int_{\Omega}(-u_k)\HH(u_j)=\int_{\Omega}(-u_k)\HH(u).
\]
It follows form the monotone convergence theorem that
\[
\lim_{j\to \infty}e_1(u_j)\leq \lim_{k\to \infty}\int_{\Omega}(-u_k)\HH(u)=e_1(u).
\]
This ends the proof.
\end{proof}

\section{Weighted energy functional}

In this section we shall define and prove some basic properties of the weighted energy functional $\EE$. This functional is sometimes called Aubin-Mabuchi energy functional.

\begin{definition}\label{defef} Let $1\leq m\leq n$, and let $\Omega$ be a bounded $m$-hyperconvex domain in $\mathbb C^n$, $n>1$. Fix $w\in \mathcal E_{1,m}$, known as the weight, and define the weighted energy functional $\EE$ by
\begin{equation}\label{energy}
\mathcal E_{1,m}\ni u\mapsto \EE(u)=\frac {1}{m+1}\sum_{j=0}^m\int_{\Omega}(u-w)(dd^cu)^j\wedge(dd^cw)^{m-j}\wedge\be \in \mathbb R.
\end{equation}
\end{definition}

We continue by proving elementary properties of $\EE$.

\begin{proposition}\label{concave}
Let $1\leq m\leq n$, and let $\Omega$ be a bounded $m$-hyperconvex domain in $\mathbb C^n$, $n>1$. Fix $w,u,v\in \mathcal E_{1,m}$, and for $t\in[0,1]$ define
\[
f(t)=\EE((1-t)u+tv).
\]
Then
\[
\begin{aligned}
f^{\prime}(t)&=\int_{\Omega}(v-u)(dd^c((1-t)u+tv))^m\wedge\be, \\
f^{\prime\prime}(t)&=m\int_{\Omega}(v-u)dd^c(u-v)\wedge(dd^c((1-t)u+tv))^{m-1}\wedge\be.
\end{aligned}
\]
In particular,  $f$ is a concave function and the derivatives of $f$ does not depend on the weight $w$.
\end{proposition}
\begin{proof} Note that
\[
\begin{aligned}
f(t)&=\EE((1-t)u+tv)=\\
&\frac {1}{m+1}\sum_{j=0}^m\int_{\Omega}(u-w+t(v-u))(dd^cu+tdd^c(v-u))^j\wedge(dd^cw)^{m-j}\wedge\be
\end{aligned}
\]
is a polynomial $(m+1)f(t)=a_0+a_1t+\dots+a_{m+1}t^{m+1}$ of degree $m+1$, where
\[
\begin{aligned}
a_0&=\sum_{j=0}^m\int_{\Omega}(u-w)(dd^cu)^j\wedge(dd^cw)^{m-j}\wedge\be;\\
a_l&=\sum_{j=l}^m\binom{j}{l}\int_{\Omega}(u-w)(dd^c(v-u))^l\wedge(dd^cu)^{j-l}\wedge(dd^cw)^{m-j}\wedge\be+\\
&\sum_{j=l-1}^m\binom{j}{l-1}\int_{\Omega}(v-u)(dd^c(v-u))^{l-1}\wedge(dd^cu)^{j-l+1}\wedge(dd^cw)^{m-j}\wedge\be;\\
a_{m+1}&=\int_{\Omega}(v-u)(dd^c(v-u))^{m}\wedge\be,
\end{aligned}
\]
for $l\geq 1$. By using integration by parts we obtain
\[
a_l=\binom{m+1}{l}\int_{\Omega}(v-u)(dd^c(v-u))^{l-1}\wedge(dd^cu)^{m-l+1}\wedge\be,
\]
and therefore
\begin{multline*}
f^{\prime}(t)=\frac {1}{m+1}(a_1+2a_2t+\dots+(m+1)a_{m+1}t^m)=\\
\frac {1}{m+1}\sum_{l=1}^{m+1}l\binom{m+1}{l}t^{l-1}\int_{\Omega}(v-u)(dd^c(v-u))^{l-1}\wedge(dd^cu)^{m-l+1}\wedge\be=\\
\sum_{l=1}^{m+1}\binom{m}{l-1}t^{l-1}\int_{\Omega}(v-u)(dd^c(v-u))^{l-1}\wedge(dd^cu)^{m-l+1}\wedge\be=\\
\int_{\Omega}(v-u)(dd^c((1-t)u+tv))^m\wedge\be.
\end{multline*}
Next,  we calculate $f^{\prime\prime}$:
\begin{multline*}
f^{\prime\prime}(t)=\sum_{j=0}^{m-1}(m-j)\binom{m}{j}t^{m-j-1}\int_{\Omega}(v-u)(dd^cu)^{j}\wedge(dd^c(v-u))^{m-j}\wedge\be=\\
\sum_{j=0}^{m-1}\binom{m-1}{j}t^{m-j-1}\int_{\Omega}(v-u)(dd^c(v-u))\wedge(dd^cu)^{j}\wedge(dd^c(v-u))^{m-j-1}\wedge\be=\\
m\int_{\Omega}(v-u)(dd^c(v-u))\wedge(dd^c((1-t)u+tv))^{m-1}\wedge\be=\\
-m\int_{\Omega}d(v-u)\wedge d^c(v-u)\wedge(dd^c((1-t)u+tv))^{m-1}\wedge\be\leq 0,
\end{multline*}
and this ends the proof.
\end{proof}

\begin{proposition}\label{basic} Let $1\leq m\leq n$, and let $\Omega$ be a bounded $m$-hyperconvex domain in $\mathbb C^n$, $n>1$. Fix $w,u,v,u_j\in \mathcal E_{1,m}$. Then
\begin{enumerate}\itemsep2mm
\item
\[
\EE(u)-\EE(v)=\frac {1}{m+1}\sum_{j=0}^m\int_{\Omega}(u-v)(dd^cu)^j\wedge(dd^cv)^{m-j}\wedge \be;
\]

\item
\[
\int_{\Omega}(u-v)\HH(u)\leq \EE(u)-\EE(v)\leq \int_{\Omega}(u-v)\HH(v);
\]

\item if $v\leq u$, then
\[
\frac {1}{m+1}\int_{\Omega}(u-v)\HH(v)\leq \EE(u)-\EE(v)\leq \int_{\Omega}(u-v)\HH(v);
\]

\item if $u_j\searrow u$, then $\EE(u_j)\searrow \EE(u)$;

\item  if $u_j\nearrow u$, where $u=(\lim_{j\to \infty}u_j)^*$, then $\EE(u_j)\nearrow \EE(u)$;

\item
\[
-\frac {1}{m+1}e_{1,m}(u+w)\leq \EE(u)\leq \frac {1}{m+1}e_{1,m}(u+w);
\]

\item $\EE(u)=-\operatorname {E}_u(w)$, $\EE(w)=0$, $\operatorname {E}_0(u)=-\frac {1}{m+1}e_{1,m}(u)$.
\end{enumerate}
\end{proposition}
\begin{proof} (1).
Let us define for $t\in[0,1]$
\[
\begin{aligned}
f_1(t)&=\EE((1-t)u+tv)-\EE(v);\\
f_2(t)&=\operatorname{E}_v((1-t)u+tv).
\end{aligned}
\]
Note that $f_1(1)=f_2(1)=0$, by Proposition~\ref{concave} we get $f_1'(t)=f_2'(t)$ and therefore
\[
\begin{aligned}
&\EE(u)-\EE(v)=f_1(0)=f_1(0)-f_1(1)=-\int_0^1f_1'(t)dt=-\int_0^1f_2'(t)dt=\\
&f_2(0)-f_2(1)=f_2(0)=\frac {1}{m+1}\sum_{j=0}^m\int_{\Omega}(u-v)(dd^cu)^j\wedge(dd^cv)^{m-j}\wedge \be.
\end{aligned}
\]
(2). By Proposition~\ref{concave} the function
\[
f(t)=\EE((1-t)v+tu), \ t\in [0,1]
\]
is concave. Therefore
\begin{multline*}
\int_{\Omega}(u-v)\HH(u)=f'(1)\leq \int_0^1f'(t)dt=f(1)-f(0)=\EE(u)-\EE(v)= \\
\int_0^1f'(t)dt\leq f'(0)=\int_{\Omega}(u-v)\HH(v).
\end{multline*}

(3). It follows from (1) and (2).

(4). It was proved by Lu in~\cite{L3} that if $u^l_j\searrow u^l$, $u_j^l,u^l\in \mathcal E_{1,m}$, $j\in \mathbb N$, $l=0,1,\dots, m$, then
\[
\int_{\Omega}u_j^0(dd^cu_j^1)\wedge\dots\wedge (dd^cu_j^m)\wedge \be\searrow \int_{\Omega}u^0(dd^cu^1)\wedge\dots\wedge (dd^cu^m)\wedge \be.
\]
Therefore, if $u_j\searrow u$, then
\[
\EE(u_j)-\EE(u)=\frac {1}{m+1}\sum_{l=0}^m\int_{\Omega}(u-u_j)(dd^cu)^l\wedge(dd^cu_j)^{m-l}\wedge \be\to 0, \ j\to \infty.
\]

(5). Let $u_j$ be an increasing sequence, and let $u=(\lim_{j\to \infty}u_j)^*$. Then by (3)
\[
0\leq \EE(u)-\EE(u_j)\leq \int_{\Omega} (u-u_j)\HH(u_j)\to 0, \ j\to \infty,
\]
by Proposition~~\ref{increasing}.

 Finally, (6) and (7) follows directly from the definition.
\end{proof}

Now let us recall the Aubin $\II$-functional. For $\psi_1,\psi_2\in \mathcal E_{1,m}$ define
\[\label{Aubin}
\II(\psi_1,\psi_2)=\int_{\Omega}(\psi_1-\psi_2)(\HH(\psi_2)-\HH(\psi_1)).
\]
Integration by part yields
\begin{equation}\label{ii}
\II(\psi_1,\psi_2)=\sum_{j=0}^{m-1}\int_{\Omega}d(\psi_1-\psi_2)\wedge d^c(\psi_1-\psi_2)\wedge(dd^c\psi_2)^{j}\wedge(dd^c\psi_1)^{m-j-1}\wedge \be.
\end{equation}

\begin{remark}
Let $w, u_j\in \mathcal E_{1,m}$, $j=0,\dots,m$ be such that
\[
\max(e_1(u_0),\dots, e_1(u_m),e_1(w))<C,
\]
for some constant $C>0$. Then we have
\begin{equation}\label{est}
\int_{\Omega}|u_0-w|(dd^cu_1)\wedge\dots\wedge(dd^cu_m)\wedge\be\leq \widetilde C,
\end{equation}
where $\tilde C$ depends only on $C$. In particular,
\[
\II(u_j,u)\leq \widetilde C.
\]
\end{remark}

\begin{proposition}\label{est2} Let $1\leq m\leq n$, and let $\Omega$ be a bounded $m$-hyperconvex domain in $\mathbb C^n$, $n>1$. Fix a constant $C>0$, and let $\varphi_1,\varphi_2, \psi_1,\psi_2\in \mathcal E_{1,m}$ be such that $e_1(\varphi_1), e_1(\varphi_2), e_1(\psi_1), e_1(\psi_2)<C$.
\begin{enumerate}

\item Then there exists a constant $D$, depending only on $C$, such that
\[
\left|\int_{\Omega}(\varphi_1-\varphi_2)(\HH(\psi_1)-\HH(\psi_2))\right|\leq D\II(\psi_1,\psi_2)^{\frac 12}.
\]

\item Then there exists a continuous increasing function $f_C:[0,\infty)\to [0,\infty)$, depending only on $C$, with $f_C(0)=0$ and such that
\[
\left|\int_{\Omega}(\varphi_1-\varphi_2)(\HH(\psi_1)-\HH(\psi_2))\right|\leq f_C\left(\II(\varphi_1,\varphi_2)\right).
\]
\end{enumerate}
\end{proposition}
\begin{proof} (1). For each $j=0,\dots, m$ and $\varphi\in \{\varphi_1,\varphi_2\}$ define
\[
\gamma_j=\int_{\Omega}\varphi(dd^c\psi_1)^j\wedge(dd^c\psi_2)^{m-j}\wedge \be.
\]
It is enough to prove that there exists a constant $C_1$, depending only on $C$, such that
\[
|\gamma_m-\gamma_0|\leq C_1 \II(\psi_1,\psi_2)^{\frac 12}.
\]
First observe that
\begin{multline*}
\gamma_{j+1}-\gamma_j=\int_{\Omega}\varphi dd^c(\psi_1-\psi_2)\wedge(dd^c\psi_1)^{j}\wedge(dd^c \psi_2)^{m-j-1}\wedge \be=\\
-\int_{\Omega}d\varphi\wedge d^c(\psi_1-\psi_2)\wedge(dd^c\psi_1)^{j}\wedge(dd^c \psi_2)^{m-j-1}\wedge \be.
\end{multline*}
From the Cauchy-Schwarz inequality it follows
\[
\left|\gamma_{j+1}-\gamma_{j}\right|^2\leq A_jB_j,
\]
where
\[
A_j=\int_{\Omega}d\varphi\wedge d^c\varphi\wedge(dd^c\psi_1)^{j}\wedge(dd^c \psi_2)^{m-j-1}\wedge \be
\]
and
\[
B_j=\int_{\Omega}d(\psi_1-\psi_2)\wedge d^c(\psi_1-\psi_2)\wedge(dd^c\psi_1)^{j}\wedge(dd^c \psi_2)^{m-j-1}\wedge \be \leq \II(\psi_1,\psi_2),
\]
and the last inequality follows from (\ref{ii}). Integration by part, and (\ref{est}), gives us
\[
A_j=-\int_{\Omega}\varphi(dd^c\varphi)\wedge(dd^c\psi_1)^{j}\wedge(dd^c \psi_2)^{m-j-1}\wedge \be\leq \tilde C,
\]
where $\tilde C$ depends only on $C$. Finally, we obtain
\[
|\gamma_m-\gamma_0|\leq |\gamma_m-\gamma_{m-1}|+\dots+|\gamma_1-\gamma_0|\leq C_1  \II(\psi_1,\psi_2)^{\frac 12},
\]
where $C_1$ depends only on $C$.

 (2). Fix $\varphi_1,\varphi_2, \psi\in \mathcal E_{1,m}$ be such that $e_1(\varphi_1), e_1(\varphi_2), e_1(\psi)<C$ and let $u=\varphi_1-\varphi_2$ and $v=\frac 12(\varphi_1+\varphi_2)$. Now define
\[
\begin{aligned}
&\alpha_j=\int_{\Omega}u(dd^c\varphi_1)^j\wedge (dd^c\psi)^{m-j}\wedge \be, \ \  j\in \{0,\dots,m\}\\
&\zeta_j=\int_{\Omega}du\wedge d^cu\wedge (dd^c v)^j\wedge (dd^c\psi)^{m-j-1}\wedge \be, \ \  j\in \{0,\dots,m-1\}.
\end{aligned}
\]
Integrating by parts gives us
\[
\begin{aligned}
&\alpha_j=\alpha_{j+1}+\int_{\Omega}u(dd^c(\psi-\varphi_1))\wedge(dd^c\varphi_1)^j\wedge (dd^c\psi)^{m-j-1}\wedge \be=\\
&\alpha_{j+1}-\int_{\Omega}du\wedge d^c(\psi-\varphi_1)\wedge(dd^c\varphi_1)^j\wedge (dd^c\psi)^{m-j-1}\wedge \be.
\end{aligned}
\]
Last integral can be estimated using the Cauchy-Schwarz inequality
\begin{multline*}
\left(\int_{\Omega}du\wedge d^c(\psi-\varphi_1)\wedge(dd^c\varphi_1)^j\wedge (dd^c\psi)^{m-j-1}\wedge \be\right)^2\leq \\
\left(\int_{\Omega}du\wedge d^cu\wedge(dd^c\varphi_1)^j\wedge (dd^c\psi)^{m-j-1}\wedge \be\right)\\
\left(\int_{\Omega}d(\psi-\varphi_1)\wedge d^c(\psi-\varphi_1)\wedge(dd^c\varphi_1)^j\wedge (dd^c\psi)^{m-j-1}\wedge \be\right)\leq
2^j\zeta_j \II(\varphi_1,\psi).
\end{multline*}
By (\ref{est}), $\II(\varphi_1,\psi)$ is bounded by a constant, depending only on $C$, and therefore
\[
|\alpha_j-\alpha_{j+1}|\leq C_1\zeta_j^{\frac 12}.
\]
Then we get
\begin{equation}\label{sum}
\left|\int_{\Omega}(\varphi_1-\varphi_2)(\HH(\varphi_1)-\HH(\psi))\right|\leq C_2\sum_{j=0}^{m-1}\zeta_j^{\frac 12},
\end{equation}
where constant $C_2$ depends only on $C$. Next, we shall estimate the constants $\zeta_j$. Again, using integration by parts
\begin{multline*}
\zeta_j=\int_{\Omega}du\wedge d^cu\wedge (dd^c v)^{j+1}\wedge (dd^c\psi)^{m-j-2}\wedge \be+\\
\int_{\Omega}du\wedge d^cu\wedge dd^c(\psi-v)\wedge(dd^c v)^j\wedge (dd^c\psi)^{m-j-2}\wedge \be=\\
\zeta_{j+1}-\int_{\Omega}du\wedge d^c(\psi-v)\wedge dd^cu\wedge (dd^c v)^j\wedge (dd^c\psi)^{m-j-2}\wedge \be.\\
\end{multline*}
Moreover by Cauchy-Schwarz inequality, and (\ref{ii}),
\begin{multline*}
\left|\int_{\Omega}du\wedge d^c(\psi-v)\wedge dd^cu\wedge (dd^c v)^j\wedge (dd^c\psi)^{m-j-2}\wedge \be\right|\leq \\
2\left|\int_{\Omega}du\wedge d^c(\psi-v)\wedge (dd^c v)^{j+1}\wedge (dd^c\psi)^{m-j-2}\wedge \be\right|\\
+2\left|-\int_{\Omega}du\wedge d^c(\psi-v)\wedge dd^c \varphi_2\wedge (dd^c v)^{j}\wedge (dd^c\psi)^{m-j-2}\wedge \be\right|\leq \\
2\zeta_{j+1}^{\frac 12}\II(\psi,v)^{\frac 12}+
2\left(\int_{\Omega}du\wedge d^cu\wedge dd^c \varphi_2\wedge (dd^c v)^{j}\wedge (dd^c\psi)^{m-j-2}\wedge \be\right)^{\frac 12}\times\\
\left(\int_{\Omega}d(\psi-v)\wedge d^c(\psi-v)\wedge dd^c \varphi_2\wedge (dd^c v)^{j}\wedge (dd^c\psi)^{m-j-2}\wedge \be\right)^{\frac 12}\\
2\zeta_{j+1}^{\frac 12}\II(\psi,v)^{\frac 12}+2\left(2\zeta_{j+1}\right)^{\frac 12}\left(2\II(\psi,v)\right)^{\frac 12}.
\end{multline*}
Therefore there exists constant $C_3$, depending only on $C$, such that
\[
\zeta_j\leq \zeta_{j+1}+C_3\zeta_{j+1}^{\frac 12}.
\]
Observe also that
\[
\zeta_{m-1}=\int_{\Omega}du\wedge d^cu\wedge (dd^c v)^{m-1}\wedge \be \leq D\II(\varphi_1,\varphi_2),
\]
where $D$ is a constant depending only on $m$. Now note that there exists a continuous increasing function $g_C:[0,\infty)\to [0,\infty)$ depending only on $C$ with $g_C(0)=0$ and such that
\[
\sum_{j=0}^{m-1}\zeta_j^{\frac 12}\leq g_C\left(\II(\varphi_1,\varphi_2)\right)
\]
and therefore by (\ref{sum})
\[
\left|\int_{\Omega}(\varphi_1-\varphi_2)(\HH(\varphi_1)-\HH(\psi))\right|\leq g_C\left(\II(\varphi_1,\varphi_2)\right).
\]
Finally,
\begin{multline*}
\left|\int_{\Omega}(\varphi_1-\varphi_2)(\HH(\psi_1)-\HH(\psi_2))\right|\leq \left|\int_{\Omega}(\varphi_1-\varphi_2)(\HH(\psi_1)-\HH(\varphi_1))\right|+\\
\left|\int_{\Omega}(\varphi_1-\varphi_2)(\HH(\varphi_1)-\HH(\varphi_2))\right|+\left|\int_{\Omega}(\varphi_1-\varphi_2)(\HH(\varphi_2)-\HH(\psi_2))\right|\leq\\
2g_C\left(\II(\varphi_1,\varphi_2)\right)+ \II(\varphi_1,\varphi_2)= f_C\left(\II(\varphi_1,\varphi_2)\right),
\end{multline*}
where $f_C(t)=t+2g_C(t)$.
\end{proof}

\section{Minimum principle}

First we need to prove the existence result for the complex Hessian type equation.

\medskip

\begin{theorem}\label{hte}Let $1\leq m\leq n$, and let $\Omega$ be a bounded $m$-hyperconvex domain in $\mathbb C^n$, $n>1$. Also, let $\mu=\HH(\varphi)$ for some $\varphi\in \mathcal E_{0,m}$. Assume also that $F(x,z)\geq 0$ is a $dx\times
d\mu $- measurable function on $(-\infty ,0]\times \Omega$ such that
\begin{enumerate}\itemsep2mm
\item for all $z\in \Omega$, the function $x\to F(x,z)$ is continuous and nondecreasing;

\item for all $x\leq 0$, the function $z\to F(x,z)$ is bounded.
\end{enumerate}
Then there exists a unique function $u\in  \mathcal E_{0,m}$  that satisfies the following  complex Hessian type equation
\[
\HH(u)=F(u(z),z)\, d\mu .
\]
\end{theorem}

\begin{proof} Since
\[
F(0,\cdot)\,d\mu\leq C\,d\mu=C\HH(\varphi),
\]
we can by~\cite{Coung} guarantee the existence of  a unique function $\psi\in \mathcal E_{0,m}$ such that $\HH(\psi)=F(0,\cdot)\,d\mu$. Set
\[
\mathcal K=\{u\in \mathcal E_{0,m}: u\geq \psi\}\, .
\]
The set $\mathcal K$ is convex, and compact in the $L^1_{loc}$-topology. Let us define a map
$\mathcal T : \mathcal K\to \mathcal K$ so that if
\[
\HH(v) =F(u(z),z)\, d\mu \, , \text{ then } \mathcal T(u)=v\, .
\]
Note that if $u\in \mathcal K$, then $F(u(z),z)\, d\mu\leq \HH(\psi)$. By~\cite{thien2} there exists a uniquely determined function  $v\in \mathcal E_{0,m}$ such that $\HH(v)=F(u(z),z)\, d\mu$, and by the comparison principle we have that $v\geq \psi$. Thus, $v\in \mathcal K$, i.e. $\mathcal T$ is well-defined.

We continue with proving that $\mathcal T$ is continuous, and then the Schauder-Tychonoff fixed point theorem concludes the existence part of this proof. Assume that $u_j \in \mathcal K$ with $u_j\to u\in \mathcal K$. By~\cite{L3} (since $\sup_{j}\int_{\Omega}(-u_j)^2d\mu<\infty$) we conclude that $u_j$ is converging to $u$ in $L^1(d\mu)$. By the stability theorem in~\cite{thien2} the sequence $v_j=\mathcal T(u_j)$ converges in capacity to some $v\in \mathcal K$. Since $v_j,v \in \mathcal K$ we  get that $\HH(v_j)$ tends to $\HH(v)$ in the weak$^*$-topology. Hence,
\[
\HH(v)=\lim_{j\to \infty}\HH(v_j)=\lim_{j\to \infty}F(u_j(z),z)\, d\mu=F(u(z),z)\, d\mu=\HH(\mathcal  T(u)),
\]
which implies that $v=\mathcal T(u)$ by the comparison principle. Thus,  $\lim_{j\to \infty}\mathcal T(u_j)=\mathcal T(u)$, i.e. $\mathcal T$ is continuous.

We now proceed with the uniqueness part. Assume that $F$ is a function that is nondecreasing in the first variable, and assume that there exist functions $u,v\in \mathcal E_{0,m}$ such that
\[
\HH(u)=F(u(z),z)\, d\mu \text{ and }\HH(v)=F(v(z),z)\, d\mu \, .
\]
On the set $\{z\in \Omega:u(z)<v(z)\}$ we have that
\[
\HH(u)=F(u(z),z)\, d\mu\leq F(v(z),z)\, d\mu=\HH(v)\, .
\]
By the comparison principle
\[
\int_{\{u<v\}}\HH(v)\leq \int_{\{u<v\}}\HH(u)\, ,
\]
hence $\HH(u)=\HH(v)$ on $\{z\in \Omega:u(z)<v(z)\}$. In a similar manner, we get that $\HH(u)=\HH(v)$ on $\{z\in \Omega:u(z)>v(z)\}$. Furthermore, on $\{u=v\}$ we have that
\[
\HH(u)=F(u(z),z)\, d\mu=F(v(z),z)\, d\mu=\HH(v)\, .
\]
Hence, $\HH(u)=\HH(v)$ on $\Omega$. Thus $u=v$.
\end{proof}

\begin{remark}
For our purpose the generality of Theorem~\ref{hte} is satisfactory. This result can be extended to more general class of measures. Similar results for the complex Monge-Amp\`ere operator was proved in~\cite{B,C,H}.
\end{remark}

\begin{definition}
For $u_1,\dots,u_k\in \mathcal E_{1,m}$ define
\[
\PP(u_1,\dots,u_k)=\Big(\sup\{\varphi\in \mathcal E_{1,m}: \varphi\leq \min(u_1,\dots,u_k)\}\Big)^*,
\]
where  $(\,)^*$ is the upper semicontinuous regularization.
\end{definition}

\begin{remark}
Note that if $u,v\in \mathcal E_{1,m}$, then $u+v\leq \PP(u,v)$ and therefore $P(u,v)\in \mathcal E_{1,m}$.
\end{remark}

\begin{theorem}\label{mp} Let $1\leq m\leq n$, and let $\Omega$ be a bounded $m$-hyperconvex domain in $\mathbb C^n$, $n>1$. Let $u,v\in \mathcal E_{1,m}$. Then the following minimum principle holds
\begin{equation}\label{min}
\HH(\PP(u,v))\leq \chi_{\{\PP(u,v)=u\}}\HH(u)+\chi_{\{\PP(u,v)=v\}}\HH(v).
\end{equation}
\end{theorem}
\begin{proof}
Let $u,v\in \mathcal E_{1,m}$. From the Cegrell-Lebesgue decomposition theorem (see~\cite{cegrell_pc, L3}) it follows that there exist $\alpha,\gamma\in \mathcal E_{0,m}$, $f,g\geq 0$, $f\in L^1_{loc}(\HH(\alpha))$, $g\in L^1_{loc}(\HH(\gamma))$ such that
\[
\HH(u)=f\HH(\alpha), \quad \text{ and } \quad \HH(v)=g\HH(\gamma).
\]
For $k\in \mathbb N$ let $u^k, v^k\in \mathcal E_{0,m}$ be the unique solutions to the following Dirichlet problems
\[
\HH(u^k)=\min(f,k)\HH(\alpha), \ \ \HH(v^k)=\min(g,k)\HH(\gamma).
\]
Furthermore, the sequences $u^k$ and $v^k$ are decreasing, and converging to $u$, and $v$, respectively.

By Theorem~\ref{hte} it follows that for any $j,k\in \mathbb N$ there exists a unique solution $\varphi_j^k\in\mathcal E_{0,m}$ to the following complex Hessian type equation:
\begin{equation}\label{4}
\HH(\varphi_j^k)=e^{j(\varphi_j^k-u^k)}\HH(u^k)+e^{j(\varphi_j^k-v^k)}\HH(v^k)=F_j(u^k_j)d\mu_{k, j},
\end{equation}
where $F_j(s)=e^{js}$ and $\mu_{k,j}=e^{-ju^k}\HH(u^k)+e^{-jv^k}\HH(v^k)$. We shall denote this equation by HTE(k,j). Let
\[
\psi^k=\PP(u^k,v^k) \quad \text { and } \quad \psi=\PP(u,v).
\]

\medskip

\emph{Claim 1}. The functions $u^k$ and $v^k$ are supersolutions to HTE(k,j). Hence, $\varphi_j^k\leq \psi^k$.

First note that $u^k$ and $v^k$ are supersolutions to HTE(k,j), since
\[
\begin{aligned}
&F_j(u^k)d\mu_{j,k}=e^{j(u^k-v^k)}\HH(v^k)+\HH(u^k)\geq \HH(u^k);\\
&F_j(v^k)d\mu_{j,k}=e^{j(v^k-u^k)}\HH(u^k)+\HH(v^k)\geq \HH(v^k).\\
\end{aligned}
\]
Using the fact that $F_j$ is increasing function and by the comparison principle we obtain
\[
\int_{\{u^k<\varphi_j^k\}}\HH(\varphi^k_j)\leq \int_{\{u^k<\varphi_j^k\}}\HH(u^k).
\]
Furthermore, on the set $\{u^k<\varphi_j^k\}$ we have
\[
\HH(u^k)\leq F_j(u^k)d\mu_{j,k}\leq F_j(\varphi^k_j)d\mu_{j,k}=\HH(\varphi^k_j),
\]
and therefore $\HH(u^k)=\HH(\varphi_j^k)$ on $\{u^k<\varphi_j^k\}$. Proposition~\ref{dp2} yields that
\[
\HH(\max(u^k,\varphi_j^k))\geq \chi_{\{u^k\geq\varphi_j^k\}}\HH(u^k)+\chi_{\{u^k<\varphi_j^k\}}\HH(\varphi_j^k)=\HH(u^k)\, ,
\]
and by the comparison principle we arrive at $u^k\geq \max(u^k,\varphi_j^k)$. Thus, $\varphi_j^k\leq u^k$. In the similar way,  one can obtain $\varphi_j^k\leq v^k$, and then $\varphi_j^k\leq \psi^k$.

\medskip

\emph{Claim 2.} The sequence $\varphi_j^k$ is increasing with respect to $j$.

 From Claim 1 we know that $\varphi_{j+1}^k\leq \psi^k$,  and then
\begin{multline*}
\HH(\varphi_{j+1}^k)=e^{\varphi_{j+1}^k-u^k}e^{j(\varphi_{j+1}^k-u^k)}\HH(u^k)+e^{\varphi_{j+1}^k-v^k}e^{j(\varphi_{j+1}^k-v^k)}\HH(v^k)\leq \\ F_{j}(\varphi_{j+1}^k)d\mu_{j,k}.
\end{multline*}
This means that $\varphi_{j+1}^k$ is a supersolution to HTE(k,j). Then again by Claim 1 we conclude that $\varphi_{j+1}^k\geq \varphi_{j}^k$.

\medskip

\emph{Claim 3.} Let $\varphi_{\infty}^k=(\lim_{j\to\infty}\varphi_{j}^k)^*$. Then it holds $\varphi_{\infty}^k=\psi^k$.

Fix $\epsilon >0$. By the comparison principle we have
\begin{multline*}
\int_{\{\varphi_{\infty}^k<\psi^k-\epsilon\}}\HH(\varphi_{j}^k)\leq \int_{\{\varphi_{j}^k<\psi^k-\epsilon\}}\HH(\varphi_{j}^k)=\\
\int_{\{\varphi_{j}^k<\psi^k-\epsilon\}}e^{j(\varphi_j^k-u^k)}\HH(u^k)+e^{j(\varphi_j^k-v^k)}\HH(v^k)\leq \\
e^{-j\epsilon}\int_{\{\varphi_{j}^k<\psi^k-\epsilon\}}\HH(u^k)+\HH(v^k)\to 0, \ j\to \infty.
\end{multline*}
Therefore, $\HH(\varphi_{\infty}^k)(\{\varphi_{\infty}^k<\psi\})=0$, and by Proposition~\ref{dp} $\varphi_{\infty}^k\geq \psi^k$. Hence, $\varphi_{\infty}^k=\psi^k$, and Claim~3 is proved.

\bigskip

Now recall that $u^k\searrow u$, $v^k\searrow v$, as $k\to \infty$, and $\psi^k=\PP(u^k,v^k)\searrow \psi=\PP(u,u)$. Fix $T>0$, and for $j>T$,  we have
\begin{multline*}
\HH(\varphi_j^k)= e^{j(\varphi_j^k-u^k)}\HH(u^k)+e^{j(\varphi_j^k-v^k)}\HH(v^k)\leq\\
e^{T(\varphi_j^k-u^k)}\HH(u^k)+e^{T(\varphi_j^k-v^k)}\HH(v^k)\leq e^{T(\varphi_j^k-u)}\HH(u)+e^{T(\varphi_j^k-v)}\HH(v),
\end{multline*}
since $\varphi_j^k\leq \PP(u^k,v^k)$. From $\HH(\varphi_j^k)$ converges to $\HH(\psi^k)$, and the dominated convergence theorem,  we get
\[
\HH(\psi^k)\leq e^{T(\psi^k-u)}\HH(u)+e^{T(\psi^k-v)}\HH(v).
\]
By letting $k\to \infty$ it follows
\[
\HH(\psi)\leq e^{T(\psi-u)}\HH(u)+e^{T(\psi-v)}\HH(v),
\]
and the desired result is obtained by letting $T\to \infty$.
\end{proof}

\begin{lemma}\label{lem1} Let $1\leq m\leq n$, and let $\Omega$ be a bounded $m$-hyperconvex domain in $\mathbb C^n$, $n>1$. Let $u,v,w\in \mathcal E_{1,m}$, $t\in [0,1]$ and let $\psi_t=\PP((1-t)u+tv,v)$. Then
\begin{equation}\label{der}
\frac {d}{dt}\EE(\psi_t)=\int_{\Omega}(v-\min(u,v))\HH(\psi_t).
\end{equation}
\end{lemma}
\begin{proof}
We shall prove~(\ref{der}) for the right derivative, the left derivative can be obtained in the similar way. For $t\in[0,1)$, let
\[
f_t=\min((1-t)u+tv,v),
\]
and let $s>0$ be small enough such that $s+t<1$. Note that
\begin{equation}\label{1}
f_{t+s}-f_t=\min((1-t-s)u+(t+s)v,v)-\min((1-t)u+tv,v)=s(v-\min(u,v)).
\end{equation}
The measure $\HH(\psi_{t+s})$ is supported on the set where $\psi_{t+s}=f_{t+s}$ (see (\ref{min})), and therefore by Proposition~\ref{basic} we have
\begin{multline}\label{eq2}
\EE(\psi_{t+s})-\EE(\psi_t)\geq \int_{\Omega}(\psi_{t+s}-\psi_t)\HH(\psi_{t+s})=\\
\int_{\Omega}(f_{t+s}-\psi_t)\HH(\psi_{t+s})\geq \int_{\Omega}(f_{t+s}-f_t)\HH(\psi_{t+s}).
\end{multline}
Assume, temporarily, that the Hessian measures $\HH(\psi_{t+s})$ tend weakly to the Hessian measure $\HH(\psi_{t})$. Using (\ref{1}), (\ref{eq2}) and that $v-\min (u,v)$ is quasi-continuous we get
\[
\lim_{s\to 0^+}\frac {\EE(\psi_{t+s})-\EE(\psi_t)}{s}\geq \int_{\Omega}(v-\min (u,v))\HH(\psi_{t}).
\]
In the similar manner, we can get the reverse inequality.  From~(\ref{min}) it follows that the measure $\HH(\psi_{t})$ is supported on $\{\psi_{t}=f_{t}\}$, and
therefore Proposition~\ref{basic} yields
\begin{multline}\label{est10}
\EE(\psi_{t+s})-\EE(\psi_t)\leq \int_{\Omega}(\psi_{t+s}-\psi_t)\HH(\psi_{t})=\\
\int_{\Omega}(\psi_{t+s}-f_t)\HH(\psi_{t})\leq \int_{\Omega}(f_{t+s}-f_t)\HH(\psi_{t}).
\end{multline}
Hence,
\[
\lim_{s\to 0^+}\frac {\EE(\psi_{t+s})-\EE(\psi_t)}{s}\leq \int_{\Omega}(v-\min (u,v))\HH(\psi_{t}).
\]
To finish the proof we have to show that $\HH(\psi_{t+s})$ tend weakly to  $\HH(\psi_{t})$. Fix $\alpha\in \mathcal E_{0,m}$, since by~\cite{L,L3} $\mathcal C^{\infty}_0(\Omega)\subset \mathcal E_{0,m}-\mathcal E_{0,m}$ it is enough to show that
\begin{equation}\label{4.7}
\int_{\Omega}\alpha (\HH(\psi_{t+s})-\HH(\psi_t))\to 0, \quad \text{ as } s\to 0^+\, .
\end{equation}
It follows from Proposition~\ref{est2} that there exists constant $D$, depending only on $e_1(\alpha)$ such that
\begin{equation}\label{4.8}
\left|\int_{\Omega}\alpha (\HH(\psi_{t+s})-\HH(\psi_t))\right|\leq D\left(\int_{\Omega}(\psi_{t+s}-\psi_{t}) (\HH(\psi_{t+s})-\HH(\psi_t))\right)^{\frac 12}.
\end{equation}
Hence (\ref{4.7}) follows from (\ref{1}), (\ref{eq2}), (\ref{est10}) and (\ref{4.8}).
\end{proof}

\section{Metric space $(\mathcal E_{1,m},d)$}

Let $1\leq m\leq n$, and let $\Omega$ be a bounded $m$-hyperconvex domain in $\mathbb C^n$. Let $\dd:\mathcal E_{1,m}\times\mathcal E_{1,m}\to \RE$ be defined by
\[
\dd (u,v)=\EE(u)+\EE(v)-2\EE(\PP(u,v)).
\]

The aim of this section is to prove that $(\mathcal E_{1,m},\dd)$ is a complete metric space.

\begin{proposition}\label{prop1} Let $1\leq m\leq n$, and let $\Omega$ be a bounded $m$-hyperconvex domain in $\mathbb C^n$, $n>1$. For all $u,v,w\in\mathcal E_{1,m}$ it holds:
\begin{enumerate}\itemsep2mm
\item $\dd(u,v)=\dd(v,u)$;
\item if $u\leq v$, then $\dd(u,v)=\EE(v)-\EE(u)$;
\item if $u\leq \psi\leq v$, then $\dd(u,v)=\dd(u,\psi)+\dd(\psi,v)$;
\item $\dd(u,v)=\dd(u,\PP(u,v))+\dd(v,\PP(u,v))$;
\item $\dd(u,v)=0$ if, and only if, $u=v$;
\item $-\EE(u)\leq \dd(u,w)$;
\item $\dd(u,0)=\frac {1}{m+1}e_1(u)$.
\end{enumerate}
\end{proposition}
\begin{proof} Properties (1), and (2), follow from the definition of $\dd$. Property (3) follows from (2) since
\begin{multline*}
\dd(u,v)=\EE(v)-\EE(u)= \\ \EE(v)-\EE(\psi)+\EE(\psi)-\EE(u)= \dd(u,\psi)+\dd(\psi,v).
\end{multline*}
To prove (4), note that by (2) we have
\[
\begin{aligned}
\dd(u,\PP(u,v))& =\EE(u)-\EE(\PP(u,v)), \text{ and }\\
\dd(v,\PP(u,v))& =\EE(v)-\EE(\PP(u,v))\, ,
\end{aligned}
\]
and therefore
\[
\dd(u,v)=\dd(u,\PP(u,v))+\dd(v,\PP(u,v)).
\]
If $u=v$, then $\dd(u,v)=0$ and the first implication of (5) is completed. On the other hand, if $\dd(u,v)=0$, then it follows from (4) that
\[
\dd(u,\PP(u,v))=\dd(v,\PP(u,v))=0\, ,
\]
and by Proposition~\ref{basic}
\begin{multline*}
\dd(u,\PP(u,v))=\EE(u)-\EE(\PP(u,v))=\\
\frac {1}{m+1}\sum_{j=0}^m\int_{\Omega}(u-\PP(u,v))(dd^cu)^j\wedge\left(dd^c\PP(u,v)\right)^{m-j}\wedge \be=0\, ,
\end{multline*}
which means that $u=\PP(u,v)$ a.e. with respect $\left(dd^c\PP(u,v)\right)^m\wedge \be$. By the domination principle, Proposition~\ref{dp}, we get $\PP(u,v)\geq u$. Hence, $u=\PP(u,v)$.
In the similar manner, we can get $v=\PP(u,v)$. Thus, $u=v$. Finally, property (6) follows from
\begin{multline*}
\dd(u,w)=\EE(u)+\EE(w)-2\EE(\PP(u,w))=\\
2(\EE(u)-\EE(\PP(u,w)))-\EE(u)\geq -\EE(u).
\end{multline*}

(7) follows from the definition of $\dd$.
\end{proof}
\begin{remark}
Note that the metric $\dd$ does not depend on the weight $w$.
\end{remark}

Now we prove some additional properties of the metric $\dd$.

\begin{lemma}\label{lem2} Let $1\leq m\leq n$, and let $\Omega$ be a bounded $m$-hyperconvex domain in $\mathbb C^n$, $n>1$. For all $u,v,\psi\in\mathcal E_{1,m}$ it holds:
\begin{enumerate}\itemsep2mm
\item $\dd(\max(u,v),u)\geq \dd(v,\PP(u,v))$;
\item $\dd(u,v)\geq \dd(\PP(u,\psi),\PP(v,\psi))$.
\end{enumerate}
\end{lemma}
\begin{proof} (1). Let us define $f=\max (u,v)$, and $g=\PP(u,v)$. Then $v\geq g$, and $f\geq u$. To prove (1) it is sufficient to show
\begin{equation}\label{3}
\EE(v)-\EE(g)\leq \EE(f)-\EE(u).
\end{equation}
Let $h_t=(1-t)u+tv$, $t\in[0,1]$. Then it holds
\[
(1-t)u+tf=(1-t)u+t\max(u,v)=\max(h_t,u).
\]
Therefore, $\{h_t>u\}=\{v>u\}$ for all $0<t<1$, and so by~\cite{L3}
\begin{multline*}
\chi_{\{v>u\}}\HH(\max(h_t,u))=\chi_{\{h_t>u\}}\HH(\max(h_t,u))=\\
\chi_{\{h_t>u\}}\HH(h_t)=\chi_{\{v>u\}}\HH(h_t).
\end{multline*}
Furthermore, $f-u=\chi_{\{v>u\}}(v-u)$, and by Proposition~\ref{concave}
\begin{multline*}
\EE(f)-\EE(u)=\int_0^1\int_{\Omega}(f-u)\HH(\max(h_t,u))=
\int_0^1\int_{\{v>u\}}(v-u)\HH(h_t).
\end{multline*}
On the other hand since $\{\PP(h_t,v)=h_t\}\subset \{h_t\leq v\}$ and $\{\PP(h_t,v)=v\}\subset \{v\leq h_t\}$, we get by~(\ref{min})
\[
\HH(\PP(h_t,v))\leq \chi_{\{h_t\leq v\}}\HH(h_t)+\chi_{\{h_t\geq v\}}\HH(v)\, .
\]
Lemma~\ref{lem1} yields
\begin{multline*}
\EE(v)-\EE(g)=\int_0^1\int_{\Omega} (v-\min(u,v))\HH(\PP(h_t,v))\leq\\
\int_0^1\int_{\{u<v\}} (v-u)\HH(h_t).
\end{multline*}

(2). Assume at the beginning that $v\leq u$. Then $v\leq \max(v,\PP(u,\psi))\leq u$, and by (1)
\[
\dd(v,u)\geq \dd(v,\max(v,\PP(u,\psi)))\geq \dd(\PP(u,\psi),\PP(\PP(u,\psi),v))=\dd(\PP(u,\psi),\PP(v,\psi))\, ,
\]
since $v\leq u$, and $\PP(\PP(u,\psi),v)=\PP(v,\psi)$. For arbitrary function $u,v$ we can repeat the above argument to obtain
\[
\begin{aligned}
\dd(u,\PP(u,v))& \geq \dd(\PP(u,\psi), \PP(u,v,\psi)), \\
\dd(v,\PP(u,v))& \geq \dd(\PP(v,\psi), \PP(u,v,\psi)).
\end{aligned}
\]
The identity $\PP(\PP(u,\psi),\PP(v,\psi))=\PP(u,v,\psi)$ implies that
\begin{multline*}
\dd(u,v)=\dd(u,\PP(u,v))+\dd(v,\PP(u,v))\geq \dd (\PP(u,\psi), \PP(u,v,\psi))
+ \\ \dd (\PP(v,\psi), \PP(u,v,\psi))=  \dd (\PP(u,\psi),\PP(\PP(u,\psi),\PP(v,\psi)))
+\\ \dd (\PP(v,\psi), \PP(\PP(u,\psi),\PP(v,\psi)))=\dd (\PP(u,\psi),\PP(v,\psi)).
\end{multline*}
\end{proof}

Now we can prove that $(\mathcal E_{1,m},\dd)$ is a metric space.

\begin{theorem}\label{thm_compmetric}
Let $1\leq m\leq n$, and let $\Omega$ be a bounded $m$-hyperconvex domain in $\mathbb C^n$, $n>1$. $(\mathcal E_{1,m},\dd)$ is a complete metric space.
\end{theorem}
\begin{proof} We have left to prove the triangle inequality, and completeness. To prove the triangle inequality let $u,v,\psi \in \mathcal E_{1,m}$, and note that in order to prove
\[
\dd(u,v)\leq \dd(u,\psi)+\dd(\psi,v)
\]
it is sufficient to show
\[
\EE(\PP(\psi,u))-\EE(\PP(u,v)))\leq \EE(\psi)-\EE(\PP(\psi,v)).
\]
From Lemma~\ref{lem2} we get
\begin{multline*}
\EE(\psi)-\EE(\PP(\psi,v))=\dd(\psi,\PP(\psi,v))\geq \dd(\PP(\psi,u),\PP(\PP(\psi,v),u))=\\
\EE(\PP(\psi,u))-\EE(\PP(\psi,v,u))\geq \EE(\PP(\psi,u))-\EE(\PP(u,v)).
\end{multline*}
The last inequality is motivated by the fact
\[
\PP(\psi,v,u)\leq \PP(u,v) \Rightarrow \EE(\PP(\psi,v,u))\leq \EE(\PP(u,v)).
\]

To prove completeness of the space $(\mathcal E_{1,m},\dd)$, let $\{u_j\}\subset \mathcal E_{1,m}$ be a Cauchy sequence. After picking a subsequence we may assume that $\dd(u_j,u_{j+1})\leq \frac {1}{2^j}$ for $j\in \mathbb N$. Let $v_{j,k}=\PP(u_j,\dots,u_{k})$, for $k\geq j$. Note that $v_{j,k}\in \mathcal E_{1,m}$ and $v_{j,k}\leq v_{j+1,k}$. Thanks to Lemma~\ref{lem2} we get
\begin{multline*}
\dd(u_j,v_{j,k})=\dd(u_j,\PP(u_j,v_{j+1,k}))=\dd(\PP(u_j,u_j),\PP(u_j,v_{j+1,k}))\leq\dd(u_j,v_{j+1,k})\leq \\
\dd(u_j,u_{j+1})+\dd(u_{j+1},v_{j+1,k}).
\end{multline*}
Iterating the above argument we obtain
\begin{equation}\label{9}
\dd(u_j,v_{j,k})\leq \sum_{s=1}^{k-j}\dd(u_{j+s-1},u_{j+s})\leq \sum_{s=1}^{k-j}\frac {1}{2^{j+s-1}}\leq \sum_{s=j}^{\infty}\frac {1}{2^s}=\frac {1}{2^{j-1}}.
\end{equation}
The sequence $v_{j,k}$ is decreasing in $k$, and it follows from (\ref{9}), that $v_j=\lim_{k\to \infty}v_{j,k}\in \mathcal E_{1,m}$, since
\begin{multline*}
\frac {1}{m+1}e_1(v_{j,k})=\dd(0,v_{j,k})\leq \dd(0,u_1)+\sum_{k=1}^{j-1}\dd(u_j,u_{j+1})+\dd(u_j,v_{j,k})\leq \\
\dd(0,u_1)+1+\frac {1}{2^{j-1}}.
\end{multline*}

Therefore by Proposition~\ref{basic} we have  $\dd(u_j,v_j)\leq \frac {1}{2^{j-1}}$. Furthermore, $v_j$ is an increasing sequence.  Set $u=(\lim_{j\to \infty}v_j)^*\in \mathcal E_{1,m}$. Proposition~\ref{basic} concludes this proof since
\[
\dd(u_j,u)\leq \dd(u_j,v_j)+\dd(v_j,u)\leq \frac {1}{2^{j-1}}+\dd(v_j,u)\to 0, \quad \text{ as } j\to \infty.
\]

\end{proof}

\section{Convergence in the space $(\mathcal E_{1,m},\dd)$}\label{sec_conv}

In this section, we prove some convergence result for the metric $\dd$.  We shall always assume that $\Omega$ is a bounded $m$-hyperconvex domain and weight $w\in \mathcal E_{1,m}$. We start with proving that metric $\dd$ is continuous under monotone sequences.

\begin{proposition}\label{monotone} Let $1\leq m\leq n$, and let $\Omega$ be a bounded $m$-hyperconvex domain in $\mathbb C^n$, $n>1$. Let $u_j, v_j\in \mathcal E_{1,m}$ be decreasing or increasing sequences converging to $u,v \in\mathcal E_{1,m}$, respectively. Then $\dd(u_j,v_j)\to \dd(u,v)$, as $j\to \infty$.
\end{proposition}
\begin{proof}
From Proposition~\ref{basic} it follows that $\EE(u_j)\to \EE(u)$, and $\EE(v_j)\to \EE(v)$, as $j\to \infty$. Hence, it is sufficient to show that $\EE(\PP(u_j,v_j))\to \EE(\PP(u,v))$, as $j\to \infty$. If $u_j$, $v_j$ are monotone sequences, then $\PP(u_j,v_j)$ is also monotone sequence.  Observe also that $\PP(u_j,v_j)\leq u_j, v_j$ and then by letting $j\to \infty$ we get
\[
(\lim_{j\to \infty}P(u_j,v_j))^*\leq u, v\, .
\]
Hence,
\[
\left(\lim_{j\to \infty}P(u_j,v_j)\right)^*\leq \PP(u,v).
\]
If the sequences are decreasing then $\PP(u_j,v_j)\geq \PP(u,v)$, and the proof is finished. Now assume that $u_j, v_j$ are increasing sequences, and let $(\lim_{j\to \infty}P(u_j,v_j))^*=\psi$. We are going to show that $\psi=\PP(u,v)$. Thanks to Theorem~\ref{mp} we get, for $k\leq j$,
\begin{multline*}
\HH(\PP(u_j,v_j))\leq \chi_{\{\psi\geq u_j\}}\HH(u_j)+\chi_{\{\psi\geq v_j\}}\HH(v_j)\leq \\
\chi_{\{\psi\geq u_k\}}\HH(u_j)+\chi_{\{\psi\geq v_k\}}\HH(v_j).
\end{multline*}
Since $\HH(u_j)\to \HH(u)$,  $\HH(v_j)\to \HH(v)$,  and $\HH(\PP(u_j,v_j))\to \HH(\psi)$, weakly as $j\to \infty$, we arrive at
\[
\HH(\psi)\leq \chi_{\{\psi\geq u_k\}}\HH(u)+\chi_{\{\psi\geq v_k\}}\HH(v).
\]
Hence, if $k\to \infty$, then
\begin{multline*}
0\leq \int_{\Omega}(\PP(u,v)-\psi)\HH(\psi)\leq \int_{\{\psi=u\}}(\PP(u,v)-\psi)\HH(u)+\\
\int_{\{\psi=v\}}(\PP(u,v)-\psi)\HH(v)\leq 0.
\end{multline*}
Now since $\HH(\psi)(\{\psi<\PP(u,v)\})=0$, then by Proposition~\ref{dp} we obtain $\PP(u,v)\leq \psi$. Hence, $\PP(u,v)=\psi$.
\end{proof}

\begin{remark}
Note that if $u_1\geq u_2$, then
\begin{equation}\label{ineq}
\II(u_1,u_2)\leq \int_{\Omega}(u_1-u_2)\HH(u_2)\leq (m+1)\dd(u_1,u_2),
\end{equation}
where $\II$ is the Aubin $\II$-functional (see page~\pageref{Aubin} for the definition).
\end{remark}

\begin{theorem}\label{thm est1} Let $1\leq m\leq n$, and let $\Omega$ be a bounded $m$-hyperconvex domain in $\mathbb C^n$, $n>1$. Fix a constant $C>0$ and let $\varphi_1,\varphi_2, \psi_1,\psi_2\in \mathcal E_{1,m}$ be such that $e_1(\varphi_1), e_1(\varphi_2), e_1(\psi_1), e_1(\psi_2)\leq C$.
\begin{enumerate}
\item Then there exists a constant $D$,  depending only on $C$, such that
\begin{equation}\label{est4}
\left|\int_{\Omega}(\varphi_1-\varphi_2)(\HH(\psi_1)-\HH(\psi_2))\right|\leq D\dd(\psi_1,\psi_2)^{\frac 12},
\end{equation}
where constant $D$ depends only on $C$.
\item Then there exists a continuous increasing function $h_C:[0,\infty)\to [0,\infty)$, depending only on $C$, with $h_C(0)=0$ and such that
\begin{equation}\label{est5}
\left|\int_{\Omega}(\varphi_1-\varphi_2)(\HH(\psi_1)-\HH(\psi_2))\right|\leq h_C\left(\dd(\varphi_1,\varphi_2)\right).
\end{equation}
\end{enumerate}

\end{theorem}
\begin{proof} From Lemma~\ref{lem2} it follows
\begin{multline*}
\frac {1}{m+1}e_1(\PP(\varphi_1,\varphi_2))=\dd(\PP(\varphi_1,\varphi_2),0)\leq \dd(\varphi_1,0)+\dd(\PP(\varphi_1,\varphi_2),\varphi_1)\leq \\
\dd(\varphi_1,0)+\dd(\varphi_2,\varphi_1)\leq \frac {3C}{m+1}.
\end{multline*}

(1). The assumptions of Proposition~\ref{est2} are all fulfilled with the constant $3C$, and therefore by (\ref{ineq})
\begin{multline*}
\left|\int_{\Omega}(\varphi_1-\varphi_2)(\HH(\psi_1)-\HH(\psi_2))\right|\leq \left|\int_{\Omega}(\varphi_1-\varphi_2)(\HH(\psi_1)-\HH(\PP(\psi_1,\psi_2)))\right|+\\
\left|\int_{\Omega}(\varphi_1-\varphi_2)(\HH(\PP(\psi_1,\psi_2))-\HH(\psi_2))\right|\leq D_1\II(\psi_1,\PP(\psi_1,\psi_2))^{\frac 12}+\\
D_2\II(\psi_2,\PP(\psi_1,\psi_2))^{\frac 12}\leq D_1(m+1)^{\frac 12}\dd(\psi_1,\PP(\psi_1,\psi_2))^{\frac 12}+\\
D_2(m+1)^{\frac 12}\dd(\psi_2,\PP(\psi_1,\psi_2))^{\frac 12}\leq \tilde D\dd(\psi_1,\psi_2)^{\frac 12},
\end{multline*}
where the constant $\tilde D$ depends only on $C$.

(2). The assumptions of Proposition~\ref{est2} are fulfilled, with the constant $3C$, and therefore by (\ref{ineq})
\begin{multline*}
\left|\int_{\Omega}(\varphi_1-\varphi_2)(\HH(\psi_1)-\HH(\psi_2))\right|\leq \left|\int_{\Omega}(\varphi_1-\PP(\varphi_1,\varphi_2))(\HH(\psi_1)-\HH(\psi_2))\right|+\\
\left|\int_{\Omega}(\PP(\varphi_1,\varphi_2)-\varphi_2)(\HH(\psi_1)-\HH(\psi_2))\right|\leq f_{3C}\left(\II(\varphi_1,\PP(\varphi_1,\varphi_2))\right)+\\
f_{3C}\left(\II(\varphi_2,\PP(\varphi_1,\varphi_2))\right)\leq f_{3C}\left((m+1)\dd(\varphi_1,\PP(\varphi_1,\varphi_2))\right)+\\
f_{3C}\left((m+1)\dd(\varphi_2,\PP(\varphi_1,\varphi_2))\right)\leq 2f_{3C}\left((m+1)\dd(\varphi_1,\varphi_2)\right)=h_{C}\left(\dd(\varphi_1,\varphi_2)\right),
\end{multline*}
where $h_C(t)=2f_{3C}((m+1)t)$.
\end{proof}

\begin{corollary}\label{cor} Let $1\leq m\leq n$, and let $\Omega$ be a bounded $m$-hyperconvex domain in $\mathbb C^n$, $n>1$. Let $u_j,u\in \mathcal E_{1,m}$. If $\lim_{j\to \infty}\dd(u_j,u)=0$,  then $u_j\to u$ in $L^1(\Omega)$, and $\HH(u_j)\to \HH(u)$ weakly, as  $j\to \infty$.
\end{corollary}
\begin{proof}
To prove the first part it is enough to apply Theorem~\ref{thm est1} (\ref{est5}) with $\psi_2=0$ and $\psi_1$ such that $\HH(\psi_1)=dV_{2n}$, where $dV_{2n}$ is the Lebesgue measure in $\mathbb C^n$.
The second part follows from (\ref{est4}), and the fact that $\mathcal C^{\infty}_0(\Omega)\subset \mathcal E_{0,m}-\mathcal E_{0,m}$ (see~\cite{cegrell_bdd}).
\end{proof}

Finally, we can prove the following characterization of the convergence with respect to the metric $\dd$.

\begin{theorem}\label{thm_conv} Let $1\leq m\leq n$, and let $\Omega$ be a bounded $m$-hyperconvex domain in $\mathbb C^n$, $n>1$. Let $u_j,u\in \mathcal E_{1,m}$. Then $\dd(u_j,u)\to 0$ as $j\to \infty$ if, and only if, $u_j\to u$ in $L^1(\Omega)$ and $\EE(u_j)\to \EE(u)$, as $j\to \infty$.
\end{theorem}
\begin{proof} First assume that $u_j\to u$ in $L^1(\Omega)$, and $\EE(u_j)\to \EE(u)$, as $j\to \infty$. Let us define $v_j=\left(\sup_{k\geq j}u_k\right)^*$, then $v_j\searrow u$, and
therefore by Proposition~\ref{basic} $\EE(v_j)\to \EE(u)$. Finally,
\[
\dd(u_j,u)\leq \dd(u_j,v_j)+\dd(v_j,u)=(\EE(v_j)-\EE(u_j))+(\EE(v_j)-\EE(u))\to 0,
\]
as $j\to \infty$. On the other hand, if $\dd(u_j,u)\to 0$, then by Corollary~\ref{cor} we get $u_j\to u$ in $L^1(\Omega)$. We have also get by Proposition~\ref{basic} and Proposition~\ref{monotone}
\begin{multline*}
\left|\EE(u_j)-\EE(u)\right|\leq \left|\EE(u_j)-\EE(v_j)\right|+\left|\EE(v_j)-\EE(u)\right|=\dd(u_j,v_j)+\dd(v_j,u)\\
\leq \dd(u_j,u)+2\dd(v_j,u)\to 0, \text{ as } j\to \infty.
\end{multline*}
\end{proof}

At the end of this section we prove that convergence in metric $\dd$ implies convergence in capacity. Let us recall that $u_j\to u$ in capacity $\operatorname{cap}_m$ if for any $K\Subset \Omega$ and any $\epsilon>0$
\[
\lim_{j\to \infty}\operatorname{cap}_m\left(K\cap\{z\in \Omega: |u_j(z)-u(z)|>\epsilon\}\right)=0,
\]
where capacity of a Borel set $A\Subset\Omega$ is defined by
\[
\operatorname{cap}_m(A)=\sup\left\{\int_{A}\HH(\varphi): \varphi\in \mathcal {SH}_m(\Omega); \ -1\leq \varphi\leq 0\right\}.
\]

\begin{proposition}\label{cap} Let $1\leq m\leq n$, and let $\Omega$ be a bounded $m$-hyperconvex domain in $\mathbb C^n$, $n>1$. If $\dd(u_j,u)\to 0$, then $u_j\to u$ in capacity $\operatorname {cap}_m$.
\end{proposition}
\begin{proof} Let us define $v_j=(\sup_{k\geq j}u_j)^*$, then $v_j\searrow u$, and $v_j\geq u_j$. Then since $v_j$ is decreasing sequence it follows from~\cite{L3} that $v_j\to u$ in capacity $\operatorname{cap}_m$. Furthermore,
\begin{multline*}
\{z\in \Omega: |u_j(z)-u(z)|>\epsilon\}\subset \left\{z\in \Omega: |v_j(z)-u_j(z)|>\frac {\epsilon}{2}\right\} \cup \\ \left\{z\in \Omega: |v_j(z)-u(z)|>\frac {\epsilon}{2}\right\}.
\end{multline*}
Therefore, it is enough to prove that
\[
\lim_{j\to \infty}\operatorname{cap}_m\left(K\cap\left\{z\in \Omega: |v_j(z)-u_j(z)|>\frac {\epsilon}{2}\right\}\right)=0.
\]
Note also that
\[
\dd(u_j,v_j)\leq \dd(u_j,u)+\dd(v_j,u)\to 0, \ j\to \infty,
\]
and therefore
\begin{multline*}
\frac {1}{m+1}\max(e_1(u_j),e_1(v_j))=\max(\dd(u_j,0),\dd(v_j,0))\leq \\
\max(\dd(u_j,u),\dd(v_j,u))+\dd(u,0)\leq C<\infty.
\end{multline*}
Fix $\psi\in \mathcal E_{0,m}$ such that $-1\leq \psi\leq 0$, and $K\Subset\Omega$. Then thanks to B\l ocki's inequality, see~\cite{L3, thien2, WanWang2}, and Proposition~\ref{est2} we arrive at
\begin{multline*}
\int_{K\cap\{z\in \Omega: |v_j(z)-u_j(z)|>\frac {\epsilon}{2}\}}\HH(\psi)\leq \\
\frac{2^{m+1}}{\epsilon ^{m+1}}\int_{K\cap\{z\in \Omega: |v_j(z)-u_j(z)|>\frac {\epsilon}{2}\}}(v_j-u_j)^{m+1}\HH(\psi)\leq \\
\frac{2^{m+1}}{\epsilon ^{m+1}}\int_{\Omega}(v_j-u_j)^{m+1}\HH(\psi)\leq \\
\frac{2^{m+1}}{\epsilon ^{m+1}}(m+1)!\|\psi\|^m_{\infty}\int_{\Omega}(v_j-u_j)\HH(u_j)\leq \\
\frac{2^{m+1}(m+1)!}{\epsilon ^{m+1}}h_{C}(\dd(u_j,v_j))\to 0, \ j\to \infty.
\end{multline*}
\end{proof}

The reverse implication in Proposition~\ref{cap} is not in general true, even in the case $m=n$. The following example is from~\cite{CAV}.

\begin{example} Let
\[
u_j(z)=\max\left(j^{\frac 1n}\ln |z|,-\frac 1j\right)
\]
be a function defined in the unit ball in $\mathbb C^n$, $n>1$. Then $u_j\in \mathcal E_0$, and
\[
(n+1)\dd(u_j,0)=e_1(u_j)=(2\pi)^n,
\]
but $u_j\to 0$ in capacity.\hfill{$\Box$}
\end{example}

\section{A comparison of topologies}\label{sec_comparison}

To study the space $\mathcal E_{1,m}$ from the normed vector space perspective we define $\delta\mathcal E_{1,m}=\mathcal E_{1,m}-\mathcal E_{1,m}$, since $\mathcal E_{1,m}$ is only a convex cone. Then for any $u\in \delta\mathcal E_{1,m}$ define
\[
\|u\|=\inf_{u_1-u_2=u \atop u_1,u_2\in \mathcal{E}_{1,m}}\left(\int_{\Omega} (-(u_1+u_2))\operatorname{H}_m(u_1+u_2)
\right)^{\frac {1}{m+1}}\, .
\]
It was proved in~\cite{thien} that $(\delta\mathcal{E}_{1,m}, \|\cdot\|)$ is a Banach space (for the case $m=n$ see~\cite{mod}). Furthermore, the cone $\mathcal{E}_{1,m}(\Omega)$ is closed in $\delta\mathcal{E}_{1,m}(\Omega)$. Recall also that if $u\in \mathcal E_{1,m}$, then $\|u\|=e_{1,m}(u)^{\frac {1}{m+1}}$. We shall in this section provide two examples that show that the norm $\|\cdot\|$, and the metric topology generated by $\dd$ are not comparable.

\begin{example}
There is no constant $C>0$ such that $\dd(u,v)\leq C\|u-v\|$. To see this take $w\in \mathcal E_{1,m}$ and $t>0$
\begin{multline*}
\dd(w+tw,tw)=\EE(tw)-\EE((1+t)w)\\ =\frac {1}{m+1}(-e_{1,m}(w))\sum_{j=0}^m((t-1)t^j-t(t+1)^j)\to \infty,
\end{multline*}
as $t\to \infty$, but
\[
\|w+tw-tw\|=\|w\|=e_{1,m}(w)^{\frac {1}{m+1}}<\infty.
\]\hfill{$\Box$}
\end{example}

\begin{example}
There is no constant $C>0$ such that $\|u-v\|\leq C\dd(u,v)$. To see this take a plurisubharmonic functions in the unit ball $\mathbb B$ in $\mathbb C^n$, $n>1$,
\[
u_j(z)=j\max (\log |z|, a_j), \ \ v_j(z)=j\max (\log |z|, c),
\]
where $a_j\searrow c$, $j\to \infty$, and $a_j,c<0$ and the sequence $a_j$ shall be specified later. By~\cite{mod} we have
\[
\|u_j-v_j\|^{n+1}=\|u_j+v_j\|^{n+1}=e_1(u_j+v_j)=(2\pi)^nj^{n+1}((a_j-c)-2^{n+1}a_j)\to \infty,
\]
as $j\to \infty$. On the other hand, note that $u_j\geq v_j$ and therefore
\begin{multline*}
\dd(u_j,v_j)=\EE(u_j)-\EE(v_j)=\frac {1}{n+1}\sum_{k=0}^n\int_{\mathbb B}(u_j-v_j)(dd^cu_j)^k\wedge(dd^cv_j)^{n-k}=\\
\frac {(2\pi)^nj^{n+1}(a_j-c)}{n+1}.
\end{multline*}
Now if we take $a_j=c+j^{-n-2}$, then we get $\dd(u_j,v_j)\to 0$, as $j\to \infty$.\hfill{$\Box$}
\end{example}

\section{Geodesics in the space $(\mathcal E_{1,m},\dd)$}\label{sec_geodesic}

Before studying geodesics in $(\mathcal E_{1,m},\dd)$, let us start recalling some elementary metric geometry. For further information see e.g.~\cite{bh}.

\begin{definition}
Let $(X,\rho)$ be a metric space. A \emph{geodesic} in $(X,\rho)$ connecting two points $x_0,x_1 \in X$ is a continuous map $f: [0,1]\to X$ such that $f(0)=x_0$, $f(1)=x_1$ and
\begin{equation}\label{geod_def}
\rho(f(t_1),f(t_2))=|t_1 -t_2|\rho(x_0,x_1),
\end{equation}
for any $t_1,t_2 \in [0,1]$.
\end{definition}

 For $1\leq m < n$, let $\widehat{\E}_{1,m}$, be the subspace $(\mathcal E_{1,m},\dd)$ of functions that are also $(m+1)$-subharmonic. These functions shall be
   defined on a $(m+1)$-hyperconvex domain $\Omega\subset\C^n$. Let $u_0,u_1\in \widehat{\E}_{1,m}$, and let us define
\begin{multline*}
SG(u_0,u_1)=\Big\{v \text{ is a } (m+1)-\text{subharmonic defined on }  \Omega\times A:\\
v\leq0,\,\limsup_{\lambda\to C_j} u(z,\lambda)\leq u_j(z), j=0,1 \Big\},
\end{multline*}
where $A=\{\lambda\in \mathbb C: 1<|\lambda|<e\}$,
\[
C_0=\{\lambda\in \mathbb C: |\lambda|=1\}, \text{ and } C_1=\{\lambda\in \mathbb C: |\lambda|=e\}.
\]
If we define,
\[
\Psi(z,\lambda)=\sup\{v(z,\lambda): v\in SG(u_0,u_1) \},
\]
then $\Psi\in SG(u_0,u_1)$, and $\Psi$ is a maximal $(m+1)$-subharmonic function defined on $\Omega\times A$. For $t\in [0,1]$,  we shall call
\[
\varphi_t(z)=\Psi(z,e^t)
\]
for a \emph{weak geodesic} joining $u_0$, and $u_1$. From~\cite{ACH} it follows that $\varphi_t$ is $m$-subharmonic function defined on $\Omega$, and then by construction $\varphi_t\in \mathcal E_{1,m}$. Furthermore, we shall also let $\widehat{\E}_{0,m}$, be the subspace $(\mathcal E_{0,m},\dd)$ of functions that are also $(m+1)$-subharmonic on $\Omega$.

\begin{proposition}\label{bp} Let $1\leq m< n$, and let $\Omega$ be a bounded $(m+1)$-hyperconvex domain in $\mathbb C^n$, $n>1$.
Assume that $u_0,v_1\in \mathcal E_{0,m+1}$, and let $\varphi_t$ be a weak geodesic joining $u_0$ and $u_1$. Then
\begin{enumerate}\itemsep2mm

\item for all $t$ function $\varphi_t\in \mathcal E_{0,m}$;

\item $\varphi_t(z)\to 0$, as $z\to \partial \Omega$;

\item $\varphi_t\geq \max (u_0-t\|u_1\|_{\infty}, u_1-(1-t)\|u_1\|_{\infty})$.

\end{enumerate}
\end{proposition}
\begin{proof}\begin{enumerate}\itemsep2mm

\item We have $u_0+u_1\leq \Psi$, so for any $t$ we get $u_0+u_1\leq \varphi_t$, so $\varphi_t\in \mathcal E_{0,m}$.

\item It follows from (1).

\item It is enough to observe that the function $\max (u_0-t\|u_1\|_{\infty}, u_1-(1-t)\|u_1\|_{\infty})\in SG(u_0,u_1)$.
\end{enumerate}
\end{proof}

\begin{remark}
Let $(\lambda,z)\in \Omega\times A$ and let $d=d_{\lambda}+d_z$, $d^c=d_{\lambda}^c+d^c_z$. Note that for any function $F(z,\lambda)$ of $(n+1)$-variables,
\begin{multline}\label{6}
(dd^cF)^{m+1}=(m+1)d_{\lambda}d_{\lambda}^cF\wedge(d_zd_z^cF)^m+\\
m(m+1)d_{\lambda}d_z^cF\wedge d_zd_{\lambda}^cF\wedge(d_zd_z^c F)^m.
\end{multline}
\end{remark}

\begin{theorem}\label{lin} Let $1\leq m< n$, and let $\Omega$ be a bounded $(m+1)$-hyperconvex domain in $\mathbb C^n$, $n>1$. Let $\varphi_0,\varphi_1\in \widehat{\mathcal E}_{0,m}$ and let $\varphi_t$, $t\in[0,1]$, be a weak geodesic joining $\varphi_0$, and $\varphi_1$. Then the functional $\EE$ is linear along $\varphi_t$ in the sense that
\[
\EE(\varphi_t)=(1-t)\EE(\varphi_0)+t\EE(\varphi_1).
\]
\end{theorem}
\begin{proof}
Observe that $\Psi(z,\lambda)=\Psi(z,|\lambda|)$, and the linearity of the function $t\to \EE(\varphi_t)$, is equivalent with the fact that the
function $\lambda \to \EE(\Psi(z,\lambda))$ is harmonic.

Our strategy is to prove that function $\Psi$ can be approximated by smooth functions $\Psi_k\in \widehat{\mathcal E}_{0,m}\cap \mathcal C^{\infty}(\Omega\times A)$, and $d_{\lambda}\Psi_k=0$ near the boundary $\partial \Omega$.

In the first step of construction note that by~\cite{ACH}, there exists a sequence $\Psi^j\in \widehat{\mathcal E}_{0,m}\cap \mathcal C^{\infty}(\Omega\times A)$ such that $\Psi^j\searrow \Psi$, $j\to \infty$. Furthermore, for all $t\in[0,1]$ we have that $\psi_t^j\searrow \varphi_t$, and then by Proposition~\ref{basic} it follows $\EE(\varphi_t^j)\to \EE(\varphi_t)$, as $j\to \infty$. Therefore, we can assume that $\Psi\in \widehat{\mathcal E}_{0,m}\cap \mathcal C^{\infty}(\Omega\times A)$.

Let $u$ be a smooth exhaustion function for the domain $\Omega$ (see~\cite{ACH}). Now fix $\epsilon >0$. In the second step of construction note that by~\cite{ACH} there exists a smooth $(m+1)$-subharmonic function $\Phi_{\epsilon}$ such that $\Phi_{\epsilon}\geq \max (\Psi-2\epsilon, \epsilon ^{-1}u)$ and $\Phi_{\epsilon}=\max (\Psi-2\epsilon, \epsilon ^{-1}u)$ outside some neighborhood of the set $\{\Psi-2\epsilon=\epsilon ^{-1}u\}$. The function $\Phi_{\epsilon}$ is of the form
\[
\Phi_{\epsilon}=\frac 12\left(\Psi-2\epsilon+\epsilon^{-1}u+\alpha_{\epsilon}(\Psi-2\epsilon-\epsilon^{-1}u)\right),
\]
where $\alpha_{\epsilon}$ is smooth convex function such that $\alpha_{\epsilon}=|x|$ for $|x|>\epsilon$. Note that $\Phi_{\epsilon}$ is equal to $\epsilon ^{-1}u$ near  $\partial \Omega\times A$, and therefore it does not depend on $\lambda$ there.  Hence, $d_{\lambda}\Psi_{\epsilon}=0$.
Let $\epsilon \to 0^+$, then $\Phi_{\epsilon}\to \Psi$ uniformly.

By the procedure described above we can assume that $\Psi$ can be approximated by a sequence $\Psi_k\in \widehat{\mathcal E}_{0,m}\cap \mathcal C^{\infty}(\Omega\times A)$ and $d_{\lambda}\Psi_k=0$ near the boundary $\partial \Omega$.

It follows from  Proposition~\ref{basic} that
\begin{equation}\label{5}
\frac {\EE(\varphi_{t+s})-\EE(\varphi_t)}{s}=\frac {1}{m+1}\sum_{k=0}^m\int_{\Omega}\left(\frac {\varphi_{t+s}-\varphi_t}{s}\right)(dd^c\varphi_{t+s})^k\wedge(dd^c\varphi_t)^{m-k}\wedge \be.
\end{equation}

Condition (\ref{5}) yields
\[
d^c_{\lambda}\EE(\Psi_k)=\int_{\Omega}d^c_{\lambda}\Psi_k\wedge(d_zd_z^c\Psi_k)^m\wedge\be.
\]

Using Stokes' theorem and the fact that $d_{\lambda}\Psi_k=0$ near the boundary $\partial \Omega$ and by (\ref{6}) we get
\begin{multline}\label{10}
d_{\lambda}d_{\lambda}^c\EE(\Psi_k)=\int_{\Omega}d_{\lambda}d^c_{\lambda}\Psi_k\wedge(d_zd_z^c\Psi_k)^m\wedge\be+\\
m\int_{\Omega}d^c_{\lambda}\Psi_k\wedge d_{\lambda}(d_zd_z^c\Psi_k)\wedge(d_zd_z^c\Psi_k)^{m-1}\wedge\be=\\
\int_{\Omega}d_{\lambda}d^c_{\lambda}\Psi_k\wedge(d_zd_z^c\Psi_k)^m\wedge\be-m\int_{\Omega}d_zd^c_{\lambda}\Psi_k\wedge d_z^cd_{\lambda}\Psi_k\wedge(d_zd_z^c\Psi_k)^{m-1}\wedge\be=\\
\frac 1{m+1}\int_{\Omega}\operatorname{H}_{m+1}(\Psi_k(z,\lambda)).
\end{multline}
The proof is finished by letting $k\to \infty$, since $\operatorname{H}_{m+1}(\Psi_k)\to \operatorname{H}_{m+1}(\Psi)=0$.
\end{proof}

Geodesics connecting different points of $(\mathcal E_{1,m},\dd)$ may not be unique, see~\cite{D1}.
\begin{theorem}\label{thm_geod} Let $1\leq m< n$, and let $\Omega$ be a bounded $(m+1)$-hyperconvex domain in $\mathbb C^n$, $n>1$. Let $\varphi_0,\varphi_1\in \widehat{\mathcal E}_{0,m}$ and let $\varphi_t$, $t\in[0,1]$ be a weak geodesic connecting $\varphi_0$ and $\varphi_1$. Then $\varphi_t$ is a geodesic in the metric space $(\mathcal E_{1,m},\dd)$.
\end{theorem}
\begin{proof} It is sufficient to prove the theorem for $s=0$ and $t>0$, i.e.
\begin{equation}\label{11}
\dd(\varphi_0,\varphi_t)=t\dd(\varphi_0,\varphi_1),
\end{equation}
since if (\ref{11}) holds, and if $\varphi_t$, $t\in[0,1]$, is a weak geodesic connecting $\varphi_0$ and $\varphi_1$, then $\varphi_s$, $s\in[0,t]$ is a weak geodesic connecting $\varphi_0$ and $\varphi_t$. Therefore,
\[
\dd(\varphi_s,\varphi_t)=\left(1-\frac st\right)\dd(\varphi_0,\varphi_t)=\left(1-\frac st\right)t\dd(\varphi_0,\varphi_1)=(t-s)\dd(\varphi_0,\varphi_1).
\]
Next, let $\psi_t$ be a weak geodesic connecting $\varphi_0$, and $\PP(\varphi_0,\varphi_1)$. Therefore, $\psi_t\leq \varphi_t$ for all $t\in [0,1]$.  Now we shall prove that $\psi_t\leq \varphi_0$ for all $t\in [0,1]$. To see it it is enough to observe that $SG(\varphi_0,\PP(\varphi_0,\varphi_1))\subset SG(\varphi_0,\varphi_0)$.

Therefore $\psi_t\leq \PP(\varphi_0,\varphi_t)$ and then by Proposition~\ref{basic} and by Theorem~\ref{lin}
\begin{multline*}
\frac 12\left(\dd(\varphi_0,\varphi_t)-t\dd(\varphi_0,\varphi_1)\right)=(1-t)\EE(\varphi_0)+t\EE(\PP(\varphi_0,\varphi_1))-\EE(\PP(\varphi_0,\varphi_t))\leq\\
(1-t)\EE(\varphi_0)+t\EE(\PP(\varphi_0,\varphi_1))-\EE(\psi_t)=0.
\end{multline*}
In a similar way, one can get
\[
\dd(\varphi_t,\varphi_1)\leq (1-t)\dd(\varphi_0,\varphi_1).
\]
Combining last two inequalities we obtain
\[
t\dd(\varphi_0,\varphi_1)\geq \dd(\varphi_0,\varphi_t)\geq \dd(\varphi_0,\varphi_1)-\dd(\varphi_t,\varphi_1)\geq t\dd(\varphi_0,\varphi_1),
\]
which means that $t\dd(\varphi_0,\varphi_1)=\dd(\varphi_0,\varphi_t)$. This ends the proof.
\end{proof}

\begin{proposition}\label{pr} Let $1\leq m< n$, and let $\Omega$ be a bounded $(m+1)$-hyperconvex domain in $\mathbb C^n$, $n>1$. Let $u_j,v_j\in \mathcal E_{1,m}$ and let $\varphi^j_t$ be a geodesic connecting $u_j$ and $v_j$. If $u_j\searrow u\in \mathcal E_{1,m}$, $v_j\searrow v\in \mathcal E_{1,m}$, then $\varphi_t^j\searrow\varphi_t$ and $\varphi_t$ is a geodesic connecting $u$ and $v$.
\end{proposition}
\begin{proof}
If $u_j$ and $v_j$ are decreasing sequences, then $\varphi_t^j$ is also decreasing for any $t\in [0,1]$. Let $s,t\in [0,1]$. Thanks to Proposition~\ref{monotone}
we get
\[
\dd(\varphi_s,\varphi_t)\leftarrow \dd(\varphi^j_s,\varphi^j_t)=|t-s|\dd(u_j,v_j)\to |t-s|\dd(u,v).
\]
\end{proof}

\begin{theorem}\label{lin2} Let $1\leq m< n$, and let $\Omega$ be a bounded $(m+1)$-hyperconvex domain in $\mathbb C^n$, $n>1$. Let $\varphi_0,\varphi_1\in \widehat{\mathcal E}_{1,m}$. Then there exists a  geodesic, $\varphi_t$, $t\in[0,1]$, connecting $\varphi_0$ and $\varphi_1$ such that the functional $\EE$ is linear along $\varphi_t$, and $\varphi_t\to \varphi_l$ in capacity as $t\to l\in \{0,1\}$.
\end{theorem}
\begin{proof} There are decreasing sequences $u_j\searrow \varphi_0$ and $v_j\searrow \varphi_1$, $j\to \infty$, $u_j,v_j\in \widehat{\mathcal E}_{0,m}$ (\cite{ACH}). For each $j$ use Theorem~\ref{lin} to get geodesics $\varphi^j_t$ connecting $u_j$ and $v_j$, and the apply Proposition~\ref{pr} to arrive at $\varphi^j_t\searrow \varphi_t$, $j\to \infty$. Furthermore, $\varphi_t$ is a geodesic connecting $\varphi_0$ and $\varphi_1$. Then by Theorem~\ref{lin}, we know that functional $\EE$ is linear along $\varphi_t$. Since,
\[
\dd(\varphi_t,\varphi_0)=t\dd(\varphi_0,\varphi_1)\to 0, \ \ t\to 0,
\]
we can use Proposition~\ref{cap} to obtain $\varphi_t\to \varphi_0$ in capacity, as $t\to 0$. In the similar manner one can prove that $\varphi_t\to \varphi_1$ in capacity as $t\to 1$.
\end{proof}

The following proposition shows that the geodesic obtaining from approximation by weak geodesic is unique.

\begin{proposition} Let $1\leq m< n$, and let $\Omega$ be a bounded $(m+1)$-hyperconvex domain in $\mathbb C^n$, $n>1$. Let $u_0,u_1,v_0,v_1\in \widehat{\mathcal E}_{0,m}$. Also let $\varphi_t$ be a weak geodesic connecting $u_0$ and $u_1$, and let $\psi_t$ be a weak geodesic connecting $v_0$ and $v_1$. Then
\[
\dd(\varphi_t,\psi_t)\leq (1-t)\dd(u_0,v_0)+t\dd(u_1,v_1).
\]
\end{proposition}
\begin{proof}
Assume that $\varphi_t$ and $\psi_t$ are weak geodesics connecting $u_0$ with $u_1$, and $v_0$ with $v_1$ respectively. Then the estimate
\[
\dd(\varphi_t,\psi_t)\leq (1-t)\dd(u_0,v_0)+t\dd(u_1,v_1).
\]
is equivalent to
\begin{equation}\label{7}
(1-t)\EE(\PP(u_0,v_0))+t\EE(\PP(u_1,v_1))\leq \EE(\PP(\varphi_t,\psi_t)).
\end{equation}
Now let $\alpha_t$ be a weak geodesic connecting $\PP(u_0,v_0)$ and $\PP(u_1,v_1)$. Then the inequality (\ref{7}) can be written in the form
\begin{equation}\label{8}
\EE(\alpha_t)\leq \EE(\PP(\varphi_t,\psi_t)).
\end{equation}
Note that
\[
\sup\{v: v\in SG(\PP(u_0,v_0),\PP(u_1,v_1))\}\leq \sup\{v: v\in SG(u_0,u_1)\}
\]
and
\[
\sup\{v: v\in SG(\PP(u_0,v_0),\PP(u_1,v_1))\}\leq \sup\{v: v\in SG(v_0,v_1)\}
\]
 and therefore $\alpha_t\leq \varphi_t$ and $\alpha_t\leq \psi_t$, so finally $\alpha_t\leq \PP(\varphi_t,\psi_t)$. Therefore inequality (\ref{8}) follows from Proposition~\ref{basic}, and this proof is finished.
\end{proof}

\end{document}